\documentclass[a4paper,notitlepage,twoside,reqno,10pt]{amsart}

\usepackage{bbm,pifont,latexsym}

\usepackage{dcolumn,indentfirst}
\usepackage[bookmarksopen=true,linktocpage=true,pdfstartview={XYZ null null 1.25}]{hyperref}
\usepackage{amsmath,amssymb,amscd,amsthm,amsfonts,mathrsfs}
\usepackage{color,colortbl,graphicx,xcolor,graphics,subfigure,extarrows,caption2}
\definecolor{mygreen}{rgb}{0.9,0.9,0.9}
\usepackage{titletoc}

\usepackage{autonum}

\newtheorem{thm}{Theorem}[section]
\newtheorem{lem}[thm]{Lemma}
\newtheorem{cor}[thm]{Corollary}
\newtheorem{prop}[thm]{Proposition}

\theoremstyle{definition}
\newtheorem*{defi}{Definition}

\newtheorem*{rmk}{Remark}
\newtheorem*{ques}{Question}

\newcommand{\EC}{\widehat{\mathbb{C}}}
\newcommand{\C}{\mathbb{C}}
\newcommand{\D}{\mathbb{D}}
\newcommand{\T}{\mathbb{T}}
\newcommand{\R}{\mathbb{R}}
\newcommand{\Z}{\mathbb{Z}}
\newcommand{\N}{\mathbb{N}}

\newcommand{\A}{\mathbb{A}}

\newcommand{\MC}{\mathcal{C}}
\newcommand{\MH}{\mathcal{H}}
\newcommand{\MI}{\mathcal{I}}
\newcommand{\MM}{\mathcal{M}}
\newcommand{\MN}{\mathcal{N}}

\newcommand{\MQ}{\mathcal{Q}}

\newcommand{\ii}{\textup{i}}
\newcommand{\id}{\textup{id}}
\newcommand{\dist}{\textup{dist}}
\newcommand{\diam}{\textup{diam}}

\newcommand{\Io}{\textup{I}}
\newcommand{\It}{\textup{II}}
\newcommand{\Is}{\textup{III}}

\newcommand{\re}{\textup{Re}}
\newcommand{\im}{\textup{Im}}

\newcommand{\Rat}{\textup{Rat}}

\makeatletter\@addtoreset{equation}{section}\makeatother 

\begin{document}

\author[WEIYUAN QIU]{WEIYUAN QIU}
\address{School of Mathematical Sciences, Fudan University, Shanghai, 200433, People's Republic of China}
\email{wyqiu@fudan.edu.cn}

\author[FEI YANG]{FEI YANG}
\address{Department of Mathematics, Nanjing University, Nanjing, 210093, People's Republic of China}
\email{yangfei@nju.edu.cn}

\title[Uniformization and dimensions of Cantor circle Julia sets]{Quasisymmetric uniformization and Hausdorff dimensions of Cantor circle Julia sets}

\begin{abstract}
For Cantor circle Julia sets of hyperbolic rational maps, we prove that they are quasisymmetrically equivalent to standard Cantor circles (i.e., connected components are round circles). This gives a quasisymmetric uniformization of all Cantor circle Julia sets of hyperbolic rational maps.

By analyzing the combinatorial information of the rational maps whose Julia sets are Cantor circles, we give a computational formula of the number of the Cantor circle hyperbolic components in the moduli space of rational maps for any fixed degree.

We calculate the Hausdorff dimensions of the Julia sets which are Cantor circles, and prove that for any Cantor circle hyperbolic component $\MH$ in the space of rational maps, the infimum of the Hausdorff dimensions of the Julia sets of the maps in $\MH$ is equal to the conformal dimension of the Julia set of any representative $f_0\in\MH$, and that the supremum of the Hausdorff dimensions is equal to $2$.
\end{abstract}



\date{\today}

\thanks{This work was supported by National Natural Science Foundation of China (grant Nos. 11731003 and 12071210) and Natural Science Foundation of Jiangsu Province (grant No. BK20191246).}


\maketitle

{\setcounter{tocdepth}{1}
\tableofcontents
}

\section{Introduction}

The study of topological and geometric properties of the Julia sets of holomorphic functions is one of the important topics in complex dynamics. In this paper we study a class of Julia sets of rational maps with special topology: they are all homeomorphic to the Cartesian product of the middle third Cantor set and the unit circle, i.e., the \textit{Cantor circles}.
McMullen is the first one who constructed such kind of Julia sets \cite{McM88}, and his family of rational maps
\begin{equation}
f_\lambda(z)=z^q+\lambda/z^p, \text{ where } q\geq 2,~p\geq 1
\end{equation}
was referred as \textit{McMullen maps} later (see \cite{DLU05}, \cite{Ste06} and \cite{QWY12}).

Besides the McMullen maps, one can find the Cantor circle Julia sets in some other families of rational maps. For example, see \cite{HP12b}, \cite{XQY14}, \cite{FY15}, \cite{QYY15}, \cite{QYY16} and \cite{WYZL19}. In particular, in the sense of topological conjugacy on the Julia sets, \textit{all} the Cantor circle Julia sets have been found in \cite{QYY15}.

Besides \cite{HP12b}, only few geometric properties were studied for the Cantor circle Julia sets. In this paper we focus our attention on the two aspects of the Cantor circle Julia sets: quasisymmetric classification and the dimensions (including Hausdorff and conformal dimensions). We will give a quasisymmetric uniformization for all hyperbolic Cantor circle Julia sets and calculate the infimum and the supremum of the Hausdorff dimensions of the Julia sets in each Cantor circle hyperbolic component. As a by-product, we obtain an explicit computational formula of the numbers of the Cantor circle hyperbolic components in the moduli space of rational maps for any fixed degree.

\subsection{Statement of the results}

Let $(X,d_X)$ and $(Y,d_Y)$ be two metric spaces. Suppose that there exist two homeomorphisms $f:X\rightarrow Y$ and $\psi:[0,+\infty)\rightarrow [0,+\infty)$ such that
\begin{equation}
\frac{d_Y(f(x),f(y))}{d_Y(f(x),f(z))}\leq \psi\Big(\frac{d_X(x,y)}{d_X(x,z)}\Big)
\end{equation}
for any distinct points $x,y,z\in X$. Then we say that $(X,d_X)$ and $(Y,d_Y)$ are \textit{quasisymmetrically equivalent} to each other.

From the topological point of view, all Cantor circle Julia sets are the same since they are all topologically equivalent (homeomorphic) to each other. Hence a natural problem is to give a uniformization of the Cantor circle Julia sets in the sense of quasisymmetric equivalence. In this paper, we prove the following result.

\begin{thm}\label{thm:quasi-unif}
Let $f$ be a hyperbolic rational map whose Julia set $J(f)$ is a Cantor circle. Then $J(f)$ is quasisymmetrically equivalent to a standard Cantor circle.
\end{thm}

The explicit definition of the ``standard" Cantor circles will be given in \S\ref{sec:quasi-unif} (see also Figure \ref{Fig:Cantor-circle-std}). Roughly speaking, a standard Cantor circle is the Cartesian product of a Cantor set and the unit circle, where this Cantor set is generated by an iterated function system whose elements are affine transformations in the logarithmic coordinate plane.
For the study of quasisymmetric uniformization of Cantor circle Julia sets of McMullen maps, one may refer to \cite{QYY18}.

\begin{figure}[!htpb]
  \setlength{\unitlength}{1mm}
  \centering
  \includegraphics[width=0.47\textwidth]{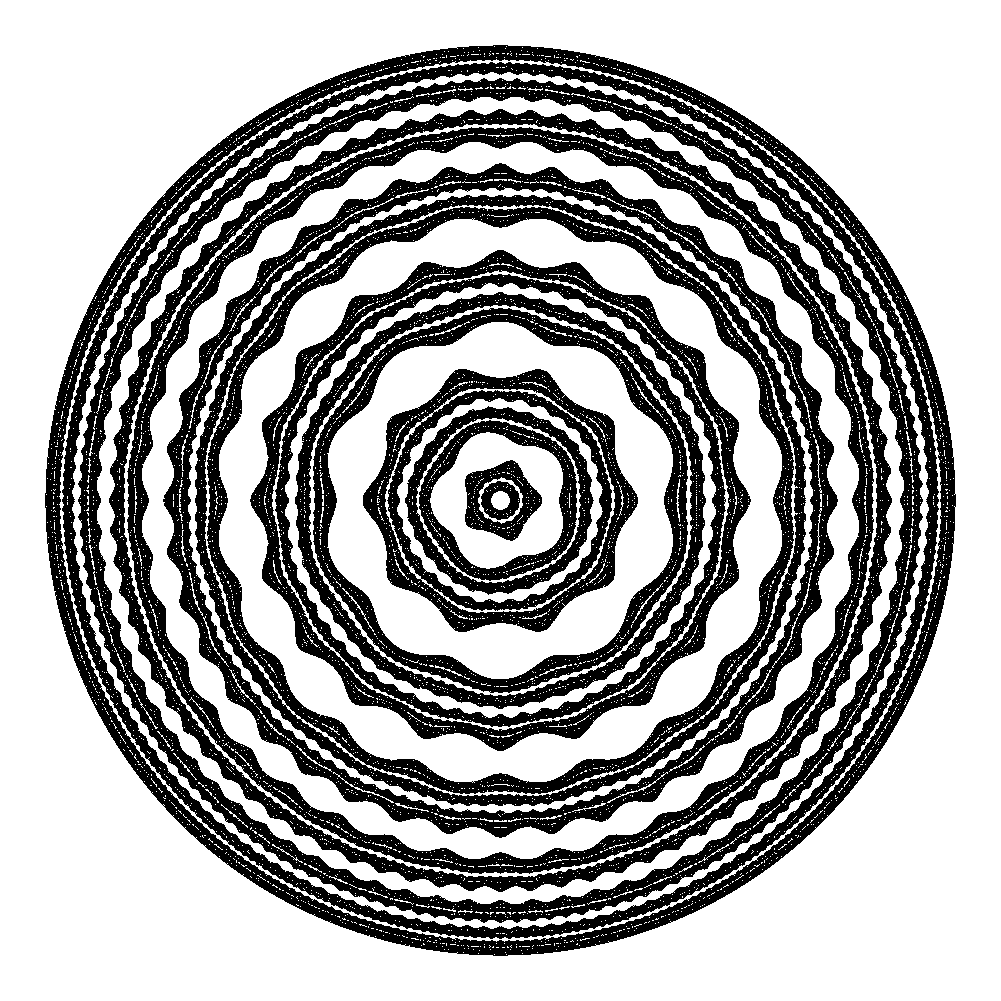}\quad
  \includegraphics[width=0.47\textwidth]{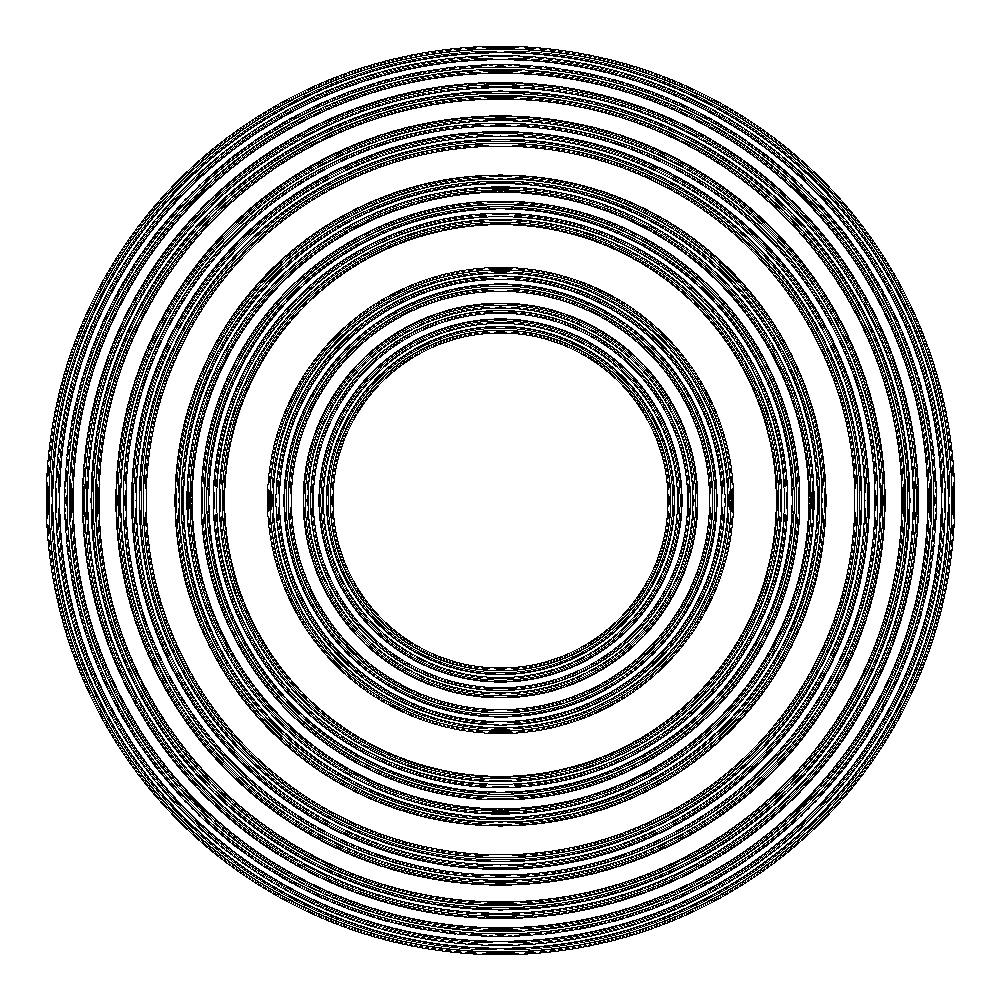}
  \caption{Left: The Julia set of the McMullen map $f(z)=z^2+10^{-5}/z^3$, which is a Cantor circle. Right: The standard Cantor circle corresponding to $f$, which is generated by the modified iterated function system $\{e^{-3}/z^3, z^2\}$ (see \S\ref{subsec-quasisym}).}
  \label{Fig:Cantor-circle-std}
\end{figure}

Recently, the quasisymmetric geometries of some other types of the Julia sets of rational maps have been studied. For example, the critically finite rational maps with Sierpi\'{n}ski carpet Julia sets was studied in \cite{BLM16}, and the corresponding results have been extended to some critically infinite cases \cite{QYZ19}. The group of all quasisymmetric self-maps of the Julia set of $z\mapsto z^2-1$ (i.e., the basilica) has been calculated in \cite{LM18} etc.

\vskip0.2cm
Let ${\rm Rat}_d=\mathbb{CP}^{2d+1}\setminus \{\textup{Resultant}=0\}$ be the space of rational maps of degree $d\geq2$.  The moduli space of $\Rat_d$ is $\mathcal{M}_d={\rm Rat}_d/{\rm PSL}_2(\mathbb{C})$, where ${\rm PSL}_2(\mathbb{C})$ is the complex projective special linear group. The M\"{o}bius conjugate class of $f\in {\rm Rat}_d$ in $\MM_d$ is denoted by $\langle f\rangle$. By abuse of notations, we also use $f$ to denote the equivalent class $\langle f\rangle$ for simplicity.
A rational map is called \textit{hyperbolic} if all its critical orbits are attracted by the attracting periodic cycles. Each connected component of all hyperbolic maps in $\MM_d$ is called a \textit{hyperbolic component}.

Let $f_1, f_2$ be two rational maps. We say that $(f_1,J(f_1))$ and $(f_2,J(f_2))$ are \emph{topologically conjugate} on their corresponding Julia sets $J(f_1)$ and $J(f_2)$ if there is an orientation preserving  homeomorphism $\phi:\EC\rightarrow\EC$ for which $\phi(J(f_1))=J(f_2)$ and $\phi\circ f_1=f_2\circ \phi$ on $J(f_1)$.
It was known from Ma\~{n}\'{e}-Sad-Sullivan \cite{MSS83} that if $f_1$ and $f_2$ are in the same hyperbolic component of $\MM_d$, then $f_1$ and $f_2$ are topologically conjugate on their corresponding Julia sets. In this paper we prove that the converse of this statement is also true when the Julia sets are Cantor circles.

\begin{thm}\label{thm:conj}
Let $f_1, f_2$ be two hyperbolic rational maps whose Julia sets are Cantor circles. Then $f_1$ and $f_2$ lie in the same hyperbolic component of $\MM_d$ if and only if they are topologically conjugate on their corresponding Julia sets.
\end{thm}

Theorem \ref{thm:conj} leads to an explicit computational formula of the number of Cantor circle hyperbolic components in $\MM_d$.

\begin{thm}\label{thm:number-CC}
The number of Cantor circle hyperbolic components in $\MM_d$ is a finite number $N(d)$ depending only on the degree $d\geq 5$, which can be calculated by
\begin{equation}\label{equ:N-d}
\begin{split}
N(d) = &~
 \sum_{n\geq 2}\,\sharp \left\{(d_1,\cdots,d_n)\in\N^n\,\left|~\sum_{i=1}^nd_i=d ~\text{and}~ \sum_{i=1}^n\frac{1}{d_i}<1\right.\right\}\\
&~ +\sum_{\textup{odd~} n\geq 3}\,\sharp
\left\{(d_1,\cdots,d_n)\in\N^n
\left|
\begin{array}{l}
\sum_{i=1}^nd_i=d,\,\sum_{i=1}^n\frac{1}{d_i}<1 \\
(d_1,\cdots,d_n)= (d_n,\cdots,d_1)
\end{array}
\right.
\right\}.
\end{split}
\end{equation}
\end{thm}

It is easy to show that the Julia set of a rational map $f$ cannot be a Cantor circle if the degree of $f$ is less than $5$ (see Proposition \ref{no-comp-intro}).
See Table \ref{Tab_number-of-new} in \S\ref{sec-No-CC} for the list of $N(d)$ with $5\leq d\leq 36$. For example, $N(5)=\sharp\{(2,3),(3,2)\}=2$, $N(6)=\sharp\{(2,4),(3,3),(4,2)\}=3$ and $N(10)=\sharp\{(2,8)$, $(3,7)$, $(4,6)$, $(5,5)$, $(6,4)$, $(7,3)$, $(8,2)\}$ $+$ $\sharp\{(3,3,4)$, $(3,4,3)$, $(4,3,3)\}$ $+$ $\sharp\{(3,4,3)\}=11$.
For a characterization of the global topological structure of Cantor circle hyperbolic components, see \cite{WY17}.

\vskip0.2cm
The \textit{conformal dimension} $\dim_C(X)$ of a compact set $X$ is the infimum of the Hausdorff dimensions of all metric spaces which are quasisymmetrically equivalent to $X$. For a given hyperbolic component $\MH$ in $\MM_d$, it follows from \cite{MSS83} that all the Julia sets of the maps in $\MH$ are quasisymmetrically equivalent to each other and hence they have the same conformal dimension. There is a following

\begin{ques}
Let $\MH$ be a hyperbolic component in $\MM_d$ with $d\geq 2$ containing a map $f_0$.
Is it true: $\inf\limits_{f\in\MH}\dim_H (J(f))=\dim_C(J(f_0))$?
\end{ques}

In this paper we give an affirmative answer to this question for Cantor circle hyperbolic components. We prove the following result.

\begin{thm}\label{thm:hdim-CC}
Let $\MH$ be a Cantor circle hyperbolic component containing a rational map $f_0$. Then
\begin{equation}
\inf_{f\in\MH}\dim_H (J(f))=\dim_C (J(f_0)) \text{\quad and\quad} \sup_{f\in\MH}\dim_H (J(f))=2.
\end{equation}
\end{thm}

In fact, we can show that the conformal dimension of $J(f_0)$ is $1+\alpha$, where $\alpha$ is the unique positive root of $\sum_{i=1}^n d_i^{-\alpha}=1$, and $(d_1,\cdots,d_n)$ is determined by the combinatorial information of the maps in the Cantor circle hyperbolic component $\MH$ (see Proposition \ref{prop:conf-dim-new}).
Moreover, we believe that $\sup_{f\in\MH}\dim_H (J(f))=2$ holds for any hyperbolic component in the space of rational maps $\textup{Rat}_d$ with any $d\geq 2$.

\medskip
Ha\"{i}ssinsky and Pilgrim constructed two quasisymmetrically inequivalent hyperbolic Cantor circle Julia sets from McMullen maps by studying their conformal dimensions \cite{HP12b}. For the study of the Hausdorff dimension of Cantor circle Julia sets (or their subsets) of McMullen maps, one may refer to \cite{WY14} and \cite[Theorem C(b)]{BW15}. For the possible range of the Hausdorff dimensions of Cantor circle Julia sets, we have the following result.

\begin{thm}\label{thm:Hdim-sharp}
The Hausdorff dimension of any Cantor circle Julia set lies in the open interval $(1,2)$. Moreover, for any given $1<s<2$, there exists a Cantor circle Julia set $J$ for which the Hausdorff dimension of $J$ is exactly $s$.
\end{thm}

Note that a Cantor circle Julia set may contain a parabolic periodic point. Hence the rational maps considered in Theorem \ref{thm:Hdim-sharp} could be hyperbolic or parabolic.

\subsection{Organization of the paper and the sketch of the proofs}

In \S\ref{sec:quasi-unif}, we divide the rational maps with Cantor circle Julia sets into three types. Each type is based on the combinations of the Cantor circle rational maps. The combinatorial information allows us to define associated iterated function systems (IFS) whose attractors are the so-called standard Cantor circles. We establish the quasisymmetric uniformization by constructing quasiconformal homeomorphisms which map the hyperbolic Cantor circle Julia sets to the attractors of the associated IFS.

\medskip
Let $f$ and $g$ be two rational maps with Cantor circle Julia sets on which the dynamics are conjugate to each other. The idea of proving Theorem \ref{thm:conj} is to make the deformations in the critical annuli and obtain continuous paths $(f_t)_{t\in[0,1]}$, $(g_t)_{t\in[0,1]}$ of hyperbolic rational maps such that $f_0=f$, $g_0=g$ and $f_1=g_1$ (see Theorem \ref{thm:conj-resta}). In order to state the procedure more clearly, the deformations are made in the standard annuli, which lie in the dynamical plane of a quasi-regular map $\widetilde{F}$ whose restriction in some annuli is exactly the IFS associated to $f$ (and $g$). This section is the most important part of this paper. As an ingredient of the proof of Theorem \ref{thm:conj-resta}, a result about the homotopic classes from annuli to disks will be established in Appendix \ref{sec:appendix}.

\medskip
Based on Theorem \ref{thm:conj}, we can obtain the computational formula of the Cantor circle hyperbolic component by considering the different topological conjugate class of Cantor circle Julia sets and hence prove Theorem \ref{thm:number-CC}. This will be done in \S\ref{sec-No-CC}.

\medskip
Still by Theorem \ref{thm:conj}, we can find a specific rational map $f_{\varrho,d_1,\cdots,d_n}$ in each Cantor circle hyperbolic component (see Theorem \ref{thm-QYY} and Corollary \ref{cor:typical-ele}). For the infimum of the Hausdorff dimensions of the Cantor circle Julia sets, we study the specific $f_{\varrho,d_1,\cdots,d_n}$, decompose the dynamics of $f_{\varrho,d_1,\cdots,d_n}$ and obtain an iterated function system. By estimating the contracting factors of the inverse of $f_{\varrho,d_1,\cdots,d_n}$ in the log-plane, we prove the first part of Theorem \ref{thm:hdim-CC} by using a modified criterion on the calculation of Hausdorff dimensions (Theorem \ref{th:Fal}). This will be done in \S\ref{sec:dim-inf}.

\medskip
For the supremum of the Hausdorff dimensions of the Cantor circle Julia sets stated in Theorem \ref{thm:hdim-CC}, we will use a theorem on Hausdorff dimensions established by Shishikura (Theorem \ref{thm-shishi}). Then the second part of Theorem \ref{thm:Hdim-sharp} can be obtained by the continuous dependence of the Hausdorff dimension of hyperbolic rational maps. The proof of the rest part of Theorem \ref{thm:Hdim-sharp} will be given in \S\ref{sec:dim-sup}.

\medskip
\noindent\textbf{Notations.}
We will use the following notations throughout the paper. Let $\mathbb{C}$ be the complex plane and $\widehat{\mathbb{C}}=\mathbb{C}\cup\{\infty\}$ the Riemann sphere. Let $\mathbb{D}_r:=\mathbb{D}(0,r)$ be the disk centered at the origin with radius $r$ and $\mathbb{T}_r:=\partial\mathbb{D}_r$ the boundary of $\mathbb{D}_r$. In particular, $\mathbb{D}:=\mathbb{D}_1$ is the unit disk and $\mathbb{T}:=\mathbb{T}_1$ is the unit circle. For $0<r<R<+\infty$, let $\mathbb{A}(r,R):=\{z\in\mathbb{C}:r<|z|<R\}$ be the annulus centered at the origin. Moreover, we denote by $\A_r:=\A(r,1)$ with $0<r<1$.

\section{Quasisymmetric uniformalization}\label{sec:quasi-unif}

From the topological point of view, all Cantor circles are the same since they are all homeomorphic to the Cartesian product of the middle third Cantor set and the unit circle. In this section we study the Cantor circle Julia sets of hyperbolic rational maps in the sense of quasisymmetric equivalence. This will give all hyperbolic Cantor circle Julia sets a more rich geometric classification.

\subsection{Combinations of Cantor circle rational maps}\label{subsec-comb}

In this subsection we give a sketch of all the possible combinations of the rational maps whose Julia sets are Cantor circles.
Let $f$ be a hyperbolic rational map of degree $d\geq 2$ whose Julia set is a Cantor set of circles. Note that the complement of any Cantor circle Julia set (i.e., the Fatou set) consists of two simply connected components and countably many doubly connected components. In the following, we always make the following

\medskip
\textit{Assumption}:  \textit{$f$ is chosen in the moduli space of rational maps such that the two simply connected Fatou components of $f$, denoted by $D_0$ and $D_\infty$, contain $0$ and $\infty$ respectively}.

\medskip
Note that all the doubly connected Fatou components of $f$ are iterated to $D_0$ or $D_\infty$ eventually.
For $n\geq 2$, let $D_1$, $\cdots$, $D_{n-1}$ be the annular components such that $f^{-1}(D_0\cup D_\infty)=D_0\cup D_\infty\cup \bigcup_{i=1}^{n-1}D_i$, where $\{D_i\}_{1\leq i\leq n-1}$ are labeled such that $D_i$ separates $D_{i'}$ from $D_{i''}$ for all $0\leq i'<i<i''\leq n-1$. The annuli $\{D_i:1\leq i\leq n-1\}$ are called \textit{critical annuli} and $\{D_i:i=0,1,\cdots,n-1,\infty\}$ are called \textit{critical Fatou components}. Let $A_i$ be the annulus (which is a closed set) between $D_{i-1}$ and $D_i$, where $1\leq i\leq n-1$ and $A_n$ the annulus between $D_{n-1}$ and $D_\infty$. Then $f^{-1}(A)=\bigcup_{i=1}^n A_i$, where $A=\EC\setminus(D_0\cup D_\infty)$. See Figure \ref{Fig:annulus}.

\begin{figure}[!htpb]
  \setlength{\unitlength}{1mm}
  \centering
  \includegraphics[width=0.95\textwidth]{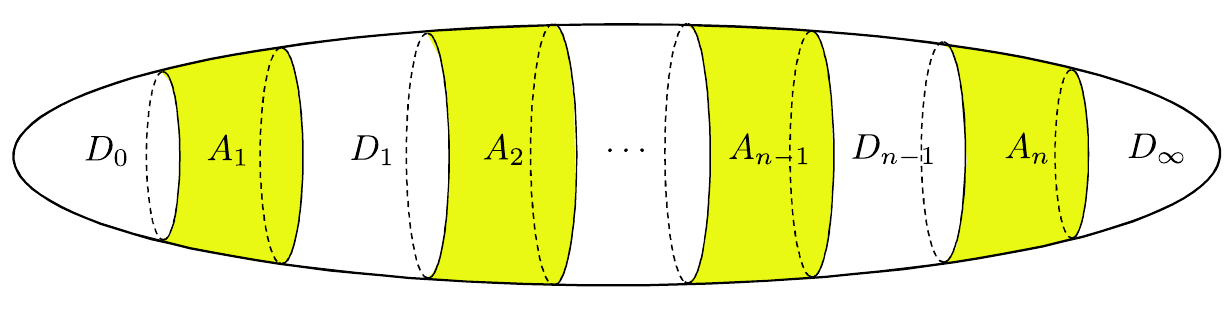}
  \caption{The structure of the Cantor circle Julia sets on the Riemann sphere. All the critical Fatou components $\{D_i:\,0\leq i\leq n-1 \text{ or } i=\infty\}$ have been marked and all the non-critical annuli $\{A_i:1\leq i\leq n\}$ have been colored by yellow.}
  \label{Fig:annulus}
\end{figure}

Note that $f|_{A_i}:A_i\to A$ is a covering map and we suppose that $\deg(f|_{A_i}:A_i\to A)=d_i$, where $1\leq i\leq n$. Then $\deg(f|_{D_i}:D_i\to D_0 \text{~or~} D_\infty)=d_i+d_{i+1}$, where $1\leq i\leq n-1$. Moreover, $\deg(f|_{D_0})=d_1$ and $\deg(f|_{D_\infty})=d_n$. Up to the conjugacy of a M\"{o}bius transformation, every rational map with Cantor circle Julia set belongs to one of the following three types.

\medskip
\textbf{Type I}: $f(D_0)=D_\infty$, $f(D_\infty)=D_\infty$ and $n\geq 2$ is even. Moreover,
\begin{equation}
f^{-1}(D_0)=\bigcup_{i=1}^{n/2} D_{2i-1} \text{ and } f^{-1}(D_\infty)=D_0\cup D_\infty\cup \bigcup_{i=1}^{(n-2)/2} D_{2i}.
\end{equation}

\vskip0.1cm
\textbf{Type II}: $f(D_0)=D_0$, $f(D_\infty)=D_\infty$ and $n\geq 3$ is odd. Moreover,
\begin{equation}
f^{-1}(D_0)=D_0\cup \bigcup_{i=1}^{(n-1)/2} D_{2i} \text{ and } f^{-1}(D_\infty)=D_\infty\cup\bigcup_{i=1}^{(n-1)/2} D_{2i-1}.
\end{equation}

\vskip0.1cm
\textbf{Type III}: $f(D_0)=D_\infty$, $f(D_\infty)=D_0$ and $n\geq 3$ is odd. Moreover,
\begin{equation}
f^{-1}(D_0)=D_\infty\cup\bigcup_{i=1}^{(n-1)/2} D_{2i-1} \text{ and } f^{-1}(D_\infty)=D_0\cup \bigcup_{i=1}^{(n-1)/2} D_{2i}.
\end{equation}

Note that $f^{-1}(A)=\bigcup_{i=1}^n A_i$ and each $A_i$ is essentially contained in $A$. It follows from Gr\"{o}tzsch's module inequality that
\begin{equation}\label{equ:combi-data}
\sum_{i=1}^n d_i=d \quad\text{and}\quad \sum_{i=1}^n\frac{1}{d_i}<1.
\end{equation}

\begin{defi}[{Combinations of Cantor circles}]
Let $\mathscr{C}$ be the collection of all the \textit{combinations} with the form $\MC=(\kappa;d_1,\cdots, d_n)$, where $\kappa\in \{\Io,\It,\Is\}$ is the \textit{type}, the array of positive integers $(d_1,\cdots, d_n)$ satisfies \eqref{equ:combi-data}, and
\begin{equation}
n\geq 2 \text{ is }\begin{cases}
\text{ even }  & \text{if } \kappa=\Io, \\
\text{ odd }   & \text{if } \kappa=\It \text{ or } \Is.
\end{cases}
\end{equation}
\end{defi}

For a hyperbolic rational map $f$ with Cantor circle Julia set, there exists at least one combinatorial data $\MC(f)=(\kappa;d_1,\cdots, d_n)\in\mathscr{C}$ corresponding to $f$.

\begin{lem}\label{lema:diff-comb}
Let $f$ be a hyperbolic rational map whose Julia set is a Cantor set of circles. Then $\MC(f)$ has exactly one element if and only if $f$ is of
\begin{itemize}
\item type $\Io$; or
\item type $\It$ or $\Is$ with $(d_1,\cdots,d_n)=(d_n,\cdots, d_1)$.
\end{itemize}
\end{lem}

\begin{proof}
Note that if $f$ has combination $(\kappa;d_1,\cdots, d_n)$ with $\kappa\in\{\It,\Is\}$, then $1/f(1/z)$ has combination $(\kappa;d_n,\cdots, d_1)$. If further $(d_1,\cdots,d_n)\neq(d_n,\cdots, d_1)$, then $\MC(f)$ consists of exactly two elements $(\kappa;d_1,\cdots, d_n)$ and $(\kappa;d_n,\cdots, d_1)$.
\end{proof}

\begin{rmk}
Actually, all the classifications and definitions in this subsection are valid for \textit{parabolic} Cantor circle Julia sets, i.e., at least one of $D_0$ and $D_\infty$ is a parabolic periodic Fatou component. However, any parabolic Cantor circle Julia set is never quasisymmetrically equivalent to the standard Cantor circles (see the definitions in the next subsection) since parabolic Cantor circle Julia sets always contain some Julia components with cusps. See \cite{QYY16}.
\end{rmk}

\subsection{Standard Cantor circles and quasisymmetric uniformalization}\label{subsec-quasisym}

We first recall the definition of iterated function systems.
Let $\Omega$ be a closed subset of $\mathbb{R}^n$ $(n\geq 1)$. The map $\psi : \Omega\rightarrow\Omega$ is called a \textit{contracting map} on $\Omega$, if there is a real number $0<c<1$ such that $|\psi (x)-\psi (y)|\leq c |x-y|$, $\forall x,y\in \Omega$. A finite family $\mathscr{F}=\lbrace \psi_1, \ldots, \psi_m \rbrace$, where $m\geq 2$, defined on $\Omega$, is called an \textit{iterated function system} (\textit{IFS} in short), if $\psi_i$ is a contracting map for all $1\leq i\leq m$. A non-empty set $J\subset\Omega$ is an \textit{attractor} of $\mathscr{F}$, if $J=\bigcup_{i=1}^m \psi_i(J)$. For any IFS, the attractor exists and is unique  (see \cite[Chap.\,9]{Fal14}).

\medskip
For each given $\MC=(\kappa;d_1,\cdots, d_n)\in\mathscr{C}$, we will define a \textit{modified} iterated function system associated to $\MC$. Let
\begin{equation}\label{equ:inteval-parti}
-1=b_1^-<b_1^+<b_2^-<b_2^+<\cdots<b_n^-<b_n^+=0
\end{equation}
be a partition of the unit interval $I=[-1,0]$, where $b_i^+-b_i^-=\frac{1}{d_i}$ for all $1\leq i\leq n$ (This is always possible since $\sum_{i=1}^n \frac{1}{d_i}<1$). For $1\leq i\leq n$, we define
\begin{equation}
L_i^\pm(x):=\pm\, d_i(x-b_i^\pm), \text{\quad where } x\in[b_i^-,b_i^+].
\end{equation}
We denote a symbol function $\chi(\pm 1):=\pm$ and define
\begin{equation}
\widetilde{\mathscr{F}}(\MC):=
\begin{cases}
\big\{\big(L_i^{\chi((-1)^i)})^{-1}:1\leq i\leq n\big\}  & \text{if } \kappa=\Io \text{ or } \Is, \\
\big\{\big(L_i^{\chi((-1)^{i-1})})^{-1}:1\leq i\leq n\big\}  & \text{if } \kappa=\It.
\end{cases}
\end{equation}
Then it is easy to see that $\widetilde{\mathscr{F}}(\MC)$ is an IFS defined on $[-1,0]$ and the attractor of $\widetilde{\mathscr{F}}(\MC)$ is a Cantor set $A(\MC)\subset [-1,0]$ having strict self-similarity.

\begin{defi}[{Standard Cantor circles}]
Let $J(\MC):=\{z\in\C:\log z\in A(\MC)\times \R\}$ be the \textit{standard Cantor circle} associated to the combination $\MC$. Then $J(\MC)$ is contained in the closed annulus $\overline{\A}(\frac{1}{e},1)$. For $1\leq i\leq n$, we define
\begin{equation}
\varphi_i^\pm(z):= z^{\pm\, d_i}/e^{\pm\, b_i^\pm d_i}:\overline{\A}(e^{b_i^-},e^{b_i^+})\to\overline{\A}(1/e,1)
\end{equation}
and
\begin{equation}
\mathscr{L}(\MC):=
\begin{cases}
\{\varphi_i^{\chi((-1)^i)}:1\leq i\leq n\}  & \text{if } \kappa=\Io \text{ or } \Is, \\
\{\varphi_i^{\chi((-1)^{i-1})}:1\leq i\leq n\}  & \text{if } \kappa=\It.
\end{cases}
\end{equation}
Note that the inverse of $\mathscr{L}(\MC)$ consists of $d=\sum_{i=1}^n d_i$ contracting maps, which form an IFS on $\overline{\A}(1/e,1)$.
By a coordinate transformation, it is straightforward to verify that $J(\MC)$ is exactly the attractor of the inverse of $\mathscr{L}(\MC)$. For convenience, we call $\mathscr{L}(\MC)$ the \textit{modified IFS}\footnote{Sometimes we omit the word ``modified" for simplicity.} associated to the combination $\MC$ and $J(\MC)$ the \textit{attractor} of $\mathscr{L}(\MC)$.
See\footnote{For the standard Cantor circle $J(\It; 4,4,4)$ in Figure \ref{Fig:Cantor-circle-scaled}, we use the partition $(-1$, $-3/4$, $-5/8$, $-3/8$, $-1/4$, $0)$ of $[-1,0]$. We will see later that each quasisymmetrically equivalent class of the standard Cantor circles depends only on the combination but not on the specific choice of the partitions. See Corollary \ref{cor:stand-CC}.} Figure \ref{Fig:Cantor-circle-scaled}.
\end{defi}

\begin{figure}[!htpb]
  \setlength{\unitlength}{1mm}
  \centering
  \includegraphics[width=0.47\textwidth]{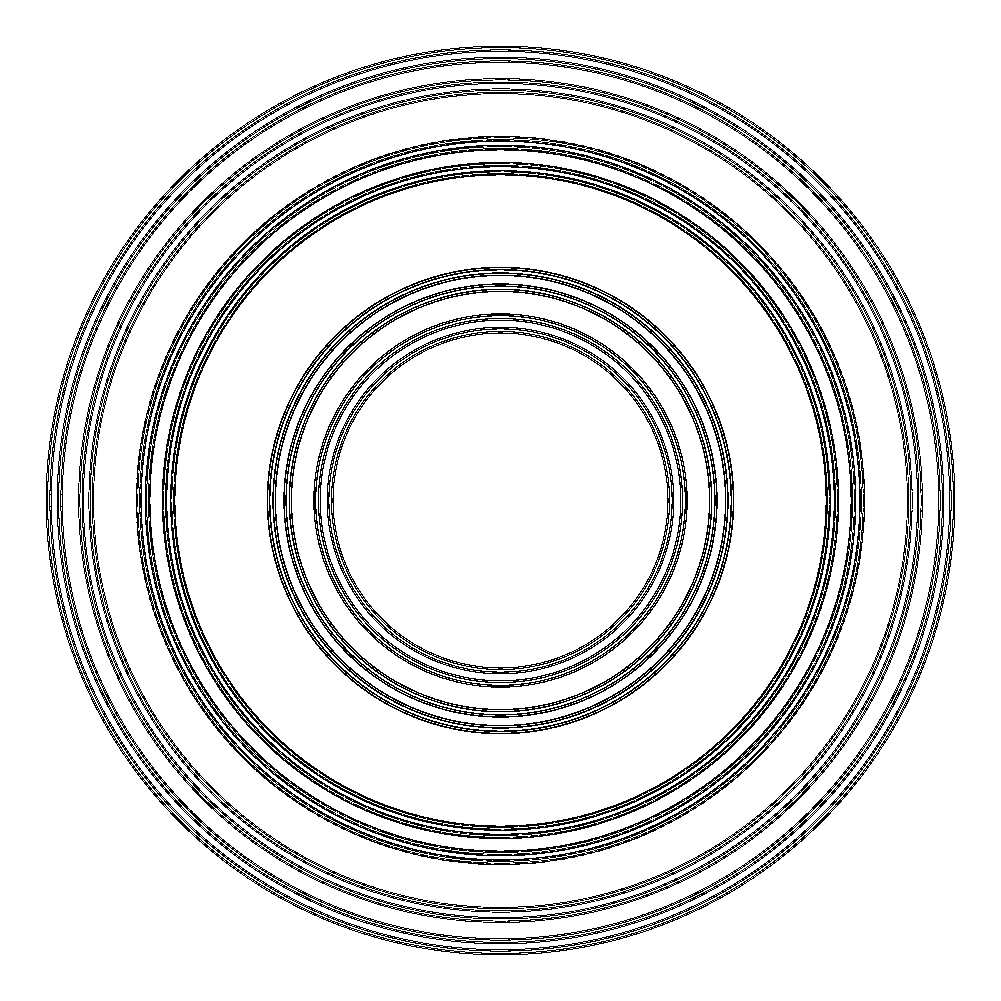}\quad
  \includegraphics[width=0.47\textwidth]{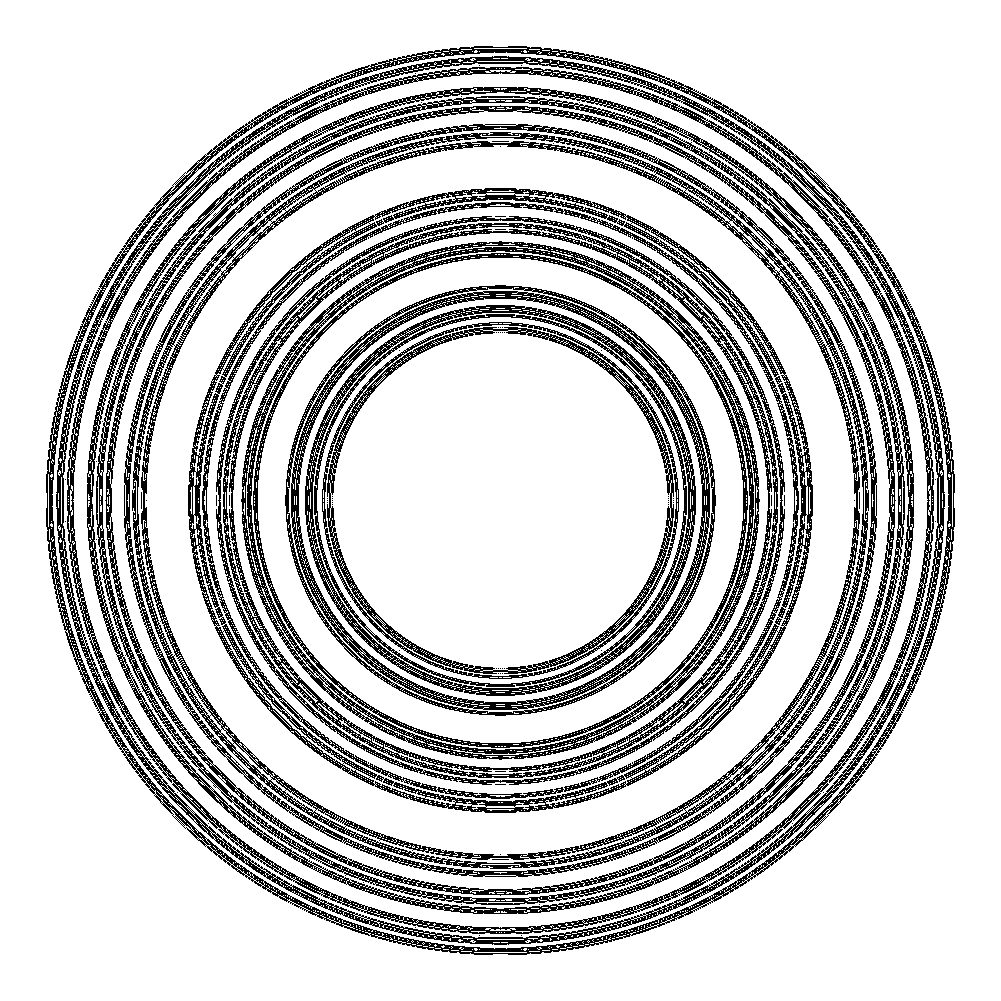}
  \caption{Two standard Cantor circles $J(\Io;3,3)$ and $J(\It; 4,4,4)$, which are generated by two modified IFS $\{z^{-3}/e^3, z^3\}$ and $\{e^3 z^4, z^{-4}/e^{5/2}, z^4\}$ respectively.}
  \label{Fig:Cantor-circle-scaled}
\end{figure}

Let $d_1$, $\cdots$, $d_n\geq 2$ be positive integers satisfying \eqref{equ:combi-data}. We use $\alpha=\alpha_{d_1,\cdots,d_n}\in(0,1)$ to denote the unique positive root of
\begin{equation}\label{equ:root}
\sum_{i=1}^n \Big(\frac{1}{d_i}\Big)^\alpha=1.
\end{equation}
According to \cite[\S7.1 and Theorem 9.3]{Fal14}, we have the following immediate result.

\begin{lem}\label{lema:Hdim}
A standard Cantor circle $J(\MC)$ with $\MC=(\kappa;d_1,\cdots,d_n)\in\mathscr{C}$ has Hausdorff dimension $1+\alpha_{d_1,\cdots,d_n}$.
\end{lem}

\begin{defi}[{Quasiregular mappings, \cite[Chap.\,1.6]{BF14}}]
Let $U$ be an open subset in $\EC$ and $1\leq K<\infty$. A continuous mapping $g:U\to\C$ is \textit{$K$-quasiregular} if and only if $g$ can be written as
\begin{equation}
g=h\circ \phi,
\end{equation}
where $\phi:U\to\phi(U)$ is $K$-quasiconformal and $h:\phi(U)\to g(U)$ is holomorphic. Equivalently, $g$ is $K$-quasiregular if and only if  $g$ is locally $K$-quasiconformal, except at a discrete set of points in $U$. The map $g$ is called \textit{$C^1$-quasiregular} if it is $K$-quasiregular for some $K\geq 1$ and also $C^1$-continuous in $U$.
\end{defi}

Now we give the quasisymmetric uniformization of the Cantor circle Julia sets of hyperbolic rational maps.

\begin{thm}\label{thm:qs-classifi}
Every Cantor circle Julia set of hyperbolic rational map is quasisymmetrically equivalent to a standard Cantor circle.
\end{thm}

\begin{proof}
Let $f$ be a hyperbolic rational map whose Julia set $J(f)$ is a Cantor circle with combinatorial data\footnote{If $\MC=\MC(f)$ consists of two elements we choose and fix any one of them (see Lemma \ref{lema:diff-comb}).} $\MC=(\kappa;d_1,\cdots, d_n)\in\mathscr{C}$. In the following we prove that $J(f)$ is quasisymmetrically equivalent to the attractor $J(\MC)$ of the modified IFS $\mathscr{L}(\MC)$. The idea is to extend the IFS $\mathscr{L}(\MC)$ to a quasiregular map $F$ and then prove that $f: J(f)\to J(f)$ is conjugated to $F:J(\MC)\to J(\MC)$ by the restriction of a quasiconformal mapping. For convenience we only prove the case $\kappa=\Io$. The cases for $\kappa=\It$, $\Is$ are completely similar.

\medskip
\textbf{Step 1}: \textit{Extending $\mathscr{L}(\MC)$ to a quasiregular map $F$}. Since $\kappa=\Io$, it means that $n\geq 2$ is even and we have
\begin{equation}\label{equ:L-MC}
\begin{split}
\mathscr{L}(\MC)
=&~\{\varphi_1^-,\varphi_2^+,\cdots,\varphi_{n-1}^-,\varphi_n^+\}\\
=&~\{z^{-d_1}/e^{d_1},z^{d_2}/e^{b_2^+ d_2},\cdots, z^{-d_{n-1}}/e^{- b_{n-1}^- d_{n-1}}, z^{d_n}\}.
\end{split}
\end{equation}
The elements in $\mathscr{L}(\MC)$ are defined by
\begin{equation}
\varphi_i^{\chi((-1)^i)}:\overline{\A}(e^{b_i^-},e^{b_i^+})\to\overline{\A}(1/e,1), \text{ where }1\leq i\leq n.
\end{equation}
Let $F:=\varphi_i^{\chi((-1)^i)}$ on $\overline{\A}(e^{b_i^-},e^{b_i^+})$, where $1\leq i\leq n$. We extend $F$ by setting
\begin{equation}\label{equ:F-qr}
F(z):=
\begin{cases}
\varphi_1^-(z)=z^{-d_1}/e^{d_1}  & \text{if } z\in\D(0,1/e), \\
\varphi_n^+(z)=z^{d_n}  & \text{if } z\in\EC\setminus\overline{\D}, \\
\text{$C^1$-quasiregular interpolation}  & \text{if } z\in\A(e^{b_i^+},e^{b_{i+1}^-}),
\end{cases}
\end{equation}
where $1\leq i\leq n-1$.
Moreover, the interpolations are chosen such that $F(\A(e^{b_i^+}$, $e^{b_{i+1}^-}))=\D(0,1/e)$ if $i$ is odd and $F(\A(e^{b_i^+},e^{b_{i+1}^-}))=\EC\setminus\overline{\D}$ if $i$ is even. Such interpolations exist indeed\footnote{Actually, McMullen maps provide a model of such kind of interpolations (from a critical annulus to a disk) when the corresponding Julia set is a Cantor circle. See \cite[\S 3]{DLU05}.}, see \cite[Lemma 7.47]{BF14} or \cite[Lemma 2.1]{PT99}. Then it is straightforward to see that $F:\EC\to\EC$ is a $C^1$-quasiregular mapping of degree $d=\sum_{i=1}^n d_i$.

Similar to the notations used in \S\ref{subsec-comb} (see also Figure \ref{Fig:annulus}), we denote $D_0':=\D(0,1/e)$, $D_\infty':=\EC\setminus\overline{\D}$, $A_i':=\overline{\A}(e^{b_i^-},e^{b_i^+})$ with $1\leq i\leq n$, and $D_i':=\A(e^{b_i^+},e^{b_{i+1}^-})$ with $1\leq i\leq n-1$. Then we have
$F(D_0')=D_\infty'$, $F(D_\infty')=D_\infty'$ and
\begin{equation}
F^{-1}(D_0')=\bigcup_{i=1}^{n/2} D_{2i-1}' \text{ and } F^{-1}(D_\infty')=D_0'\cup D_\infty'\cup \bigcup_{i=1}^{(n-2)/2} D_{2i}'.
\end{equation}

\textbf{Step 2}: \textit{Construction of a sequence of quasiconformal mappings}. Since $f$ is hyperbolic, it is known that the Julia components of $f$ are all quasicircles (see \cite[Corollary 1.7]{QYY16}). In particular, the boundaries $\partial D_0$, $\partial D_\infty$ and all the connected components of $f^{-1}(\partial D_0\cup\partial D_\infty)$ are quasi-circles. There exists a quasiconformal mapping $\phi_0:\EC\rightarrow\EC$ satisfying $\phi_0(D_0)=D_0'$ and $\phi_0(D_\infty)=D_\infty'$. Hence $\phi_0(\partial D_0)=\partial D_0'$ and $\phi_0(\partial D_\infty)=\partial D_\infty'$. Moreover, $\phi_0$ can be chosen such that $\phi_0\circ f=F\circ\phi_0$ on $\partial D_0\cup \partial D_\infty$.

Since both $f:A_i\to A$ and $F:A_i'\to A'$ are covering mappings of degree $d_i$, where $1\leq i\leq n$, there exists a lift $\phi_1:A_i\to A_i'$, which is quasiconformal\footnote{Usually a quasiconformal map is defined in a domain. Here we mean that $\phi_1:A_i\to A_i'$ is the restriction of a quasiconformal map defined in an open annulus containing $A_i$.}, such that the following diagram is commutative:
\begin{equation}
\begin{CD}
A_i @>\phi_1>> A_i' \\
@VVf V   @VVF V \\
A @>\phi_0>> A'.
\end{CD}
\end{equation}
Note that $\phi_0\circ f=F\circ\phi_0$ on $\partial D_0\cup \partial D_\infty$. One can choose $\phi_1$ such that $\phi_1|_{\partial D_0}=\phi_0|_{\partial D_0}$ and $\phi_1|_{\partial D_\infty}=\phi_0|_{\partial D_\infty}$. The choices of the lifts $\phi_1:A_i\to A_i'$ for $2\leq i\leq n-1$ are not unique. We fix one choice of them.

\medskip
Define $\phi_1:=\phi_0$ on $D_0\cup D_\infty$. Then $\phi_1$ is defined on $\EC$ except in $\bigcup_{i=1}^{n-1}D_i$. Since all components of $f^{-1}(\partial D_0\cup\partial D_\infty)$ are quasicircles, one can extend $\phi_1$ continuously to the annuli $\{D_i\}_{1\leq i\leq n-1}$ by $\phi_1:D_i\to D_i'$, to obtain a quasiconformal mapping $\phi_1:\EC\to\EC$ such that \begin{itemize}
\item $\phi_1|_A$ is homotopic to $\phi_0|_A$ rel $\partial A=\partial D_0\cup\partial D_\infty$;
\item $\phi_0\circ f =F\circ \phi_1$ on $\bigcup_{i=1}^n A_i$; and
\item $\phi_1\circ f =F\circ \phi_1$ on $f^{-1}(\partial D_0\cup \partial D_\infty)$.
\end{itemize}

Now we define $\phi_2$. First, let $\phi_2|_{D_i}=\phi_1|_{D_i}$ for $i\in\{0,1,\cdots,n-1,\infty\}$. Since $\phi_1|_A$ is homotopic to $\phi_0|_A$ rel $\partial A$, it follows that there exist lifts $\phi_2:A_i\to A_i'$ of $\phi_1:A\to A'$ satisfying $\phi_1\circ f=F\circ\phi_2$, where $1\leq i\leq n$, such that $\phi_2:\EC\to\EC$ is continuous and $\phi_2|_A$ is homotopic to $\phi_1|_A$ rel $\partial A$. In particular, $\phi_2:\EC\to\EC$ is a quasiconformal mapping which satisfies
\begin{itemize}
\item The dilatation of $\phi_2$ satisfies $K(\phi_2)=K(\phi_1)$;
\item $\phi_2(z)=\phi_1(z)$ for all $z\in f^{-1}(D_0\cup D_\infty)$;
\item $\phi_1\circ f=F\circ\phi_2$ on $\bigcup_{i=1}^n A_i$; and
\item $\phi_2\circ f=F\circ\phi_2$ on $f^{-2}(\partial D_0\cup \partial D_\infty)$.
\end{itemize}

Suppose that we have obtained $\phi_{k-1}$ for some $k\geq 2$, then $\phi_k$ can be obtained completely similarly to the procedure above. Inductively, one can obtain a sequence of quasiconformal mappings $\{\phi_k:\EC\to\EC\}_{k\geq 0}$ such that for all $k\geq 1$, the following results hold:
\begin{itemize}
\item $K(\phi_k)=K(\phi_1)$;
\item $\phi_k(z)=\phi_{k-1}(z)$ for $z\in f^{-(k-1)}(D_0\cup D_\infty)$;
\item $\phi_{k-1}\circ f=F\circ\phi_k$ on $\bigcup_{i=1}^n A_i$; and
\item $\phi_k\circ f=F\circ\phi_k$ on $f^{-k}(\partial D_0\cup \partial D_\infty)$.
\end{itemize}

\smallskip
\textbf{Step 3}: \textit{The limit conjugates the dynamics on the Julia set to that on the attractor}. One can see that the sequence $\{\phi_k:\EC\to\EC\}_{k\geq 0}$ forms a normal family. Taking any convergent subsequence of $\{\phi_k:\EC\to\EC\}_{k\geq 0}$, we denote the limit by $\phi_\infty$. Then $\phi_\infty:\EC\to\EC$ is a quasiconformal mapping satisfying $\phi_\infty\circ f=F\circ\phi_\infty$ on $\bigcup_{k\geq 0} f^{-k}(\partial D_0\cup \partial D_\infty)$. Since $\phi_\infty$ is continuous, it follows that $\phi_\infty\circ f=F\circ\phi_\infty$ holds on the closure of $\bigcup_{k\geq 0}f^{-k}(\partial D_0\cup \partial D_\infty)$, which is the Julia set of $f$. Since $\phi_\infty(J(f))=J(\MC)$, this implies that $J(f)$ is quasisymmetrically equivalent to $J(\MC)$. The proof of Theorems \ref{thm:qs-classifi} and \ref{thm:quasi-unif} is finished.
\end{proof}

\begin{rmk}
If one uses the theory of combinatorial equivalence (see Appendix A in \cite{McM98b} for further details), then the proof of Theorem \ref{thm:qs-classifi} can be largely simplified. However, we present such detailed and more direct proof here since we need to use the following observations in the next section.

(1) In Step 2, $\phi_0:\EC\to\EC$ can be chosen such that it is $C^1$-continuous (even smooth) in $D_0\cup D_\infty$ since near $\partial D_0$ (resp. $\partial D_\infty$) $f$ is conformally conjugate to $z^{-d_1}$ (resp. $z^{d_n}$). Similarly, $\phi_1:D_i\to D_i'$ can be chosen such that it is $C^1$-continuous for all $1\leq i\leq n-1$. Then by definition, $\phi_\infty$ is $C^1$-continuous in $f^{-1}(D_0\cup D_\infty)$ since $\phi_k(z)=\phi_{k-1}(z)$ for all $k\geq 1$ and all $z\in f^{-(k-1)}(D_0\cup D_\infty)$.

(2) Since for all $k\geq 1$, one has $\phi_k(z)=\phi_{k-1}(z)$ for $z\in f^{-(k-1)}(D_0\cup D_\infty)$, $\phi_{k-1}\circ f=F\circ\phi_k$ holds on $\bigcup_{i=1}^n A_i$ and $\phi_k\circ f=F\circ\phi_k$ holds on $f^{-k}(\partial D_0\cup \partial D_\infty)$, it follows that $\phi_\infty\circ f(z)=F\circ\phi_\infty(z)$ holds for all $z\in \bigcup_{i=1}^n A_i$.
\end{rmk}

As an immediate corollary of Theorem \ref{thm:qs-classifi}, we have the following special result (see \cite[Theorem 1.1(b)]{QYY18}):

\begin{cor}
If the Julia set $J_\lambda$ of the McMullen map $f_\lambda(z)=z^q+\lambda/z^p$ is a Cantor circle, then $J_\lambda$ is quasisymmetrically equivalent to the standard Cantor circle $J(\Io;p,q)$, which is the attractor generated by the IFS $\{z^{-p}/e^p, z^q\}$.
\end{cor}

For each given combination $\MC\in\mathscr{C}$, the definition of the standard Cantor circle $J(\MC)$ depends on the partition of the unit interval $[-1,0]$ (if $n\geq 3$). See \eqref{equ:inteval-parti}. However, from the proof of Theorem \ref{thm:qs-classifi} we have the following immediate result.

\begin{cor}\label{cor:stand-CC}
All standard Cantor circles with the same combination $\MC\in\mathscr{C}$ (the partitions of $[-1,0]$ in \eqref{equ:inteval-parti} are allowed to be different) are in the same quasisymmetrically equivalent class.
\end{cor}

From Corollary \ref{cor:stand-CC} we know that the classes of quasisymmetrically equivalent Cantor circles are determined by the combinatorial data but not the geometric information. 






\section{Topological conjugacy and hyperbolic components}

In order to find \textit{all} rational maps (in the sense of topological conjugacy on the Julia sets) whose Julia sets are Cantor circles, the following Theorem \ref{thm-QYY} was proved in \cite{QYY15}.

\begin{thm}\label{thm-QYY}
For any $\varrho\in\{0,1\}$ and $n\geq 2$ positive integers $d_1,\cdots,d_n$ satisfying $\sum_{i=1}^{n}\frac{1}{d_i}<1$, there are parameters $a_1,\cdots,a_{n-1}$ such that the Julia set of
\begin{equation}\label{family-QYY}
f_{\varrho,d_1,\cdots,d_n}(z)=z^{(-1)^{n-\varrho} d_1}\prod_{i=1}^{n-1}(z^{d_i+d_{i+1}}-a_i^{d_i+d_{i+1}})^{(-1)^{n-i-\varrho}}
\end{equation}
is a Cantor circle. Moreover, any rational map whose Julia set is a Cantor circle must be topologically conjugate to $f_{\varrho,d_1,\cdots,d_n}$ for some $\varrho$ and $d_1,\cdots,d_n$ on their corresponding Julia sets.
\end{thm}


Theorem \ref{thm-QYY} gives a complete topological classification of the Cantor circle Julia sets of rational maps under the dynamical behaviors.
To study the hyperbolic components of Cantor circle type, we hope to find a representative map with the form \eqref{family-QYY} in each Cantor circle hyperbolic component. This is one of the motivations to prove the following result.

\begin{thm}\label{thm:conj-resta}
Let $f$, $g$ be two hyperbolic rational maps with the same degree $d$ whose Julia sets are Cantor circles on which they are topologically conjugate. Then $f$ and $g$ lie in the same hyperbolic component of the moduli space $\MM_d$.
\end{thm}

\begin{proof}
The proof will be divided into several steps. Since $f$ and $g$ are conjugate on their Julia sets, they have the same combinatorial data. Without loss of generality, we assume that they have the same combination $\MC=(\Io;d_1,\cdots,d_n)\in\mathscr{C}$. The rest two types of combinations can be treated completely similarly. The idea of the proof can be summed up as following: For $f$ we assume that the attracting cycle is super-attracting. Then we prove that $f$ is quasiconformally conjugated to a quasiregular map $\widetilde{F}$ whose restriction on some annuli is exactly the IFS $\mathscr{L}(\MC)$ (see the definition in \S\ref{subsec-quasisym}). Next we deform the map $\widetilde{F}$ and construct a continuous path $(\widetilde{F}_t)_{t\in[0,1]}$ of quasiregular maps such that $\widetilde{F}_0=\widetilde{F}$ and $\widetilde{F}_1=F$, where $F:\EC\to\EC$ is the quasiregular map defined in \eqref{equ:F-qr}. From this one can obtain a continuous path $(f_t)_{t\in[0,1]}$ of hyperbolic rational maps such that $f_0=f$ and $f_1=\xi_1\circ F\circ\xi_1^{-1}$ for some quasiconformal mapping $\xi_1:\EC\to\EC$.

Similarly, the same construction guarantees the existence of continuous path $(g_t)_{t\in[0,1]}$ of hyperbolic rational maps such that $g_0=g$ and $g_1=\xi_1\circ F\circ\xi_1^{-1}=f_1$ (Careful: the quasiconformal map $\xi_1:\EC\to\EC$ corresponding to $f_1$ and to $g_1$ is the same!). Note that the map $F$ here is the same map as in the previous paragraph. Then the theorem follows since one can connect $f$ with $g$ by a continuous path in the hyperbolic component. Now we make the proof precisely.

\medskip
\textbf{Step 1}: \textit{Transferring attracting to super-attracting (multi-critical to unicritical)}. Let $\MH$ be the Cantor circle hyperbolic component containing $f$. According to \cite[Chap.\,4]{BF14}, by performing a standard quasiconformal surgery, there exists a continuous path in $\MH$ connecting $f$ with $\widehat{f}$, such that $\widehat{f}$ has a super-attracting basin $D_\infty$ with super-attracting fixed point $\infty$ in which $\widehat{f}_1$ is conjugate to $z\mapsto z^{d_n}$, and moreover, $\widehat{f}:D_0\setminus\{0\}\to D_\infty\setminus\{\infty\}$ is a covering map of degree $d_1$ (note that $\widehat{f}$ has the same combination as $f$).

\medskip
\textbf{Step 2}: \textit{From rational maps to quasiregular maps}. For saving the notations, we assume that the given $f$ is exactly $\widehat{f}$. We continue using the notations, such as $D_i$, $A_i$, $A$ and etc, for a rational map (and hence $f$) with Cantor circle Julia set as in \S\ref{subsec-comb}.
Since the combination $\MC$ is of type I, it implies that $n\geq 2$ is even and
\begin{equation}
\mathscr{L}(\MC)=\{\varphi_1^-,\varphi_2^+,\cdots,\varphi_n^+\}=\{z^{-d_1}/e^{d_1},z^{d_2}/e^{b_2^+ d_2},\cdots,z^{d_n}\}
\end{equation}
is the IFS defined in \eqref{equ:L-MC}. At this step we construct a quasiconformal conjugacy between $f$ and a quasiregular map $\widetilde{F}:\EC\to\EC$, such that the restriction of $\widetilde{F}$ on the union of the annuli $\bigcup_{i=1}^n\overline{\A}(e^{b_i^-},e^{b_i^+})$ is exactly the IFS $\mathscr{L}(\MC)$, where $(b_i^\pm)_{i=1}^n$ are numbers given in \eqref{equ:inteval-parti}. As before, we denote $D_0':=\D(0,1/e)$, $D_\infty':=\EC\setminus\overline{\D}$, $A_i':=\overline{\A}(e^{b_i^-},e^{b_i^+})$ with $1\leq i\leq n$, and $D_i':=\A(e^{b_i^+},e^{b_{i+1}^-})$ with $1\leq i\leq n-1$.

Let $F:\EC\to\EC$ be the $C^1$-quasiregular map\footnote{Note that in the proof of this theorem, $F$ is seen to be fixed.} defined in \eqref{equ:F-qr}. Then $F$ is $C^1$-continuous on $\EC$. There exists a quasiconformal mapping $\phi_0:\EC\rightarrow\EC$ such that
\begin{itemize}
\item $\phi_0(D_0)=D_0'$, $\phi_0(D_\infty)=D_\infty'$;
\item $\phi_0\circ f(z)=F\circ\phi_0(z)$ for all $z\in D_0\cup D_\infty$; and
\item $\phi_0$ is conformal in $D_0$ and $D_\infty$ (in fact they are B\"{o}ttcher coordinates).
\end{itemize}
By using a completely similar argument as in the proof of Theorem \ref{thm:qs-classifi}, one can obtain a sequence of quasiconformal mappings $(\phi_k:\EC\to\EC)_{k\geq 0}$ such that for all $k\geq 1$, the following statements hold:
\begin{itemize}
\item the dilatation of $\phi_k$ satisfies $K(\phi_k)=K(\phi_1)$;
\item $\phi_k(z)=\phi_{k-1}(z)$ for $z\in f^{-(k-1)}(D_0\cup D_\infty)$;
\item $\phi_{k-1}\circ f=F\circ\phi_k$ on $\bigcup_{i=1}^n A_i$; and
\item $\phi_k\circ f=F\circ\phi_k$ on $f^{-k}(\partial D_0\cup \partial D_\infty)$.
\end{itemize}
Note that $(\phi_k:\EC\to\EC)_{k\geq 0}$ is a normal family. Taking a convergent subsequence of $(\phi_k)_{k\geq 0}$ whose limit is denoted by $\phi_\infty:\EC\to\EC$, we have $\phi_\infty\circ f=F\circ\phi_\infty$ on $D_0\cup D_\infty\cup\bigcup_{i=1}^n A_i$. The map
\begin{equation}
\widetilde{F}(z):=\phi_\infty\circ f\circ\phi_\infty^{-1}(z):\EC\to\EC
\end{equation}
is quasiregular and $\widetilde{F}=F$ on $D_0'\cup D_\infty'\cup\bigcup_{i=1}^n A_i'$. In particular, the restriction of $\widetilde{F}$ on $\bigcup_{i=1}^n A_i'$ is exactly the IFS $\mathscr{L}(\MC)$.

By the construction of $\phi_n$ (see the remark following the proof of Theorem \ref{thm:qs-classifi}), we can choose the sequence $\{\phi_k\}_{k\in\N}$ such that the limit $\phi_\infty:\EC\to\EC$ is $C^1$-continuous in $\bigcup_{i=1}^{n-1}D_i$. This implies that $\widetilde{F}$ is holomorphic in $\EC\setminus \bigcup_{i=1}^{n-1}\overline{D}_i'$ and $C^1$-continuous in $\bigcup_{i=1}^{n-1}D_i'$.


\medskip
\textbf{Step 3}: \textit{Partial-twist deformations in the annuli}. Although both $\widetilde{F}$ and $F$ are quasiregular extensions of the IFS $\mathscr{L}(\MC)$ on $\bigcup_{i=1}^n A_i'$, $\widetilde{F}$ needs not to be homotopic to $F$ rel $\partial D_i'$ for some $1\leq i\leq n-1$. Indeed, for $1\leq i\leq n-1$, it turns out that (see Lemma \ref{lem:homotopy-Cui} in Appendix \ref{sec:appendix}) there exists $k_i'\in\Z$ such that $\widetilde{F}\circ T_{d_{i+1}}^{\circ k_i'}|_{\overline{D}_i'}$ is homotopic to $F|_{\overline{D}_i'}$ rel $\partial D_i'$, where
\begin{equation}
T_{d_{i+1}}(z)=z e^{\frac{2\pi\ii}{d_{i+1}}\frac{|z|-s_i}{r_i-s_i}}
\end{equation}
is a \textit{partial-twist} map along $\overline{D}_i'$, $z\in \overline{D}_i'=\overline{\A}(s_i,r_i)$, $s_i=e^{b_i^+}$ and $r_i=e^{b_{i+1}^-}$.

\medskip
Recall that $n\geq 2$ is even. In the following we assume that $n\geq 4$ since the argument for case $n=2$ is completely similar and easier. Define $k_1:=k_1'$.
For every $t\in[0,1]$, we define a family of mappings by setting
\begin{equation}
F_t^1(z):=
\begin{cases}
\widetilde{F}(z e^{-k_1\frac{2\pi\ii}{d_2}\frac{|z|-s_1}{r_1-s_1} t})  & \text{if } z\in D_1', \\
\widetilde{F}(z e^{-k_1\frac{2\pi\ii}{d_2} t})  & \text{if } z\in A_2', \\
\widetilde{F}(z e^{-k_1\frac{2\pi\ii}{d_2}\frac{r_2-|z|}{r_2-s_2} t})  & \text{if } z\in D_2', \\
\widetilde{F}(z)  & \text{other}.
\end{cases}
\end{equation}
It is straightforward to verify that $F_t^1:\EC\to\EC$ is quasiregular, holomorphic in $\EC\setminus \bigcup_{i=1}^{n-1}\overline{D}_i'$ for every $t\in[0,1]$, $F_t^1$ depends continuously on $t\in[0,1]$ and $\partial F_t^1(z)/\partial z$, $\partial F_t^1(z)/\partial \overline{z}$ depend continuously on $t\in[0,1]$ for every $z\in\EC\setminus\bigcup_{j=1}^{n-1}\partial D_j'$. Moreover,
\begin{itemize}
\item $F_0^1=\widetilde{F}$;
\item $F_1^1=\widetilde{F}$ on $\EC\setminus (D_1'\cup D_2')$;
\item $F_1^1|_{\overline{D}_1'}$ is homotopic to $F|_{\overline{D}_1'}$ rel $\partial D_1'$; and
\item there exists $k_2\in\Z$ such that $F_1^1\circ T_{d_3}^{\circ k_2}|_{\overline{D}_2'}$ is homotopic to $F|_{\overline{D}_2'}$ rel $\partial D_2'$.
\end{itemize}

Inductively, for $2\leq i\leq n-2$ and $t\in[0,1]$, we define
\begin{equation}
F_t^i(z):=
\begin{cases}
F_1^{i-1}(z e^{-k_i\frac{2\pi\ii}{d_{i+1}}\frac{|z|-s_i}{r_i-s_i} t})  & \text{if } z\in D_i', \\
F_1^{i-1}(z e^{-k_i\frac{2\pi\ii}{d_{i+1}} t})  & \text{if } z\in A_{i+1}', \\
F_1^{i-1}(z e^{-k_i\frac{2\pi\ii}{d_{i+1}}\frac{r_{i+1}-|z|}{r_{i+1}-s_{i+1}} t})  & \text{if } z\in D_{i+1}', \\
F_1^{i-1}(z)  & \text{other},
\end{cases}
\end{equation}
where $\{k_{i+1}\in\Z:2\leq i\leq n-2\}$ is determined as following: when $F_t^i$ is defined, there exists $k_{i+1}\in\Z$ such that $F_1^i\circ T_{d_{i+2}}^{\circ k_{i+1}}|_{\overline{D}_{i+1}'}$ is homotopic to $F|_{\overline{D}_{i+1}'}$ rel $\partial D_{i+1}'$. It is easy to see that $F_t^i:\EC\to\EC$ is quasiregular, holomorphic in $\EC\setminus \bigcup_{i=1}^{n-1}\overline{D}_i'$ for every $t\in[0,1]$, $F_t^i$ depends continuously on $t\in[0,1]$ and $\partial F_t^i(z)/\partial z$, $\partial F_t^i(z)/\partial \overline{z}$ depend continuously on $t\in[0,1]$ for every $z\in\EC\setminus\bigcup_{j=1}^{n-1}\partial D_j'$. Moreover,
\begin{itemize}
\item $F_0^i=F_1^{i-1}$;
\item $F_1^i=\widetilde{F}$ on $\EC\setminus \bigcup_{j=1}^{i+1}D_j'$; and
\item $F_1^i|_{\overline{D}_j'}$ is homotopic to $F|_{\overline{D}_j'}$ rel $\partial D_j'$ for all $1\leq j\leq i$.
\end{itemize}

For $t\in[0,1]$, we define
\begin{equation}
F_t^{n-1}(z):=
\begin{cases}
F_1^{n-2}(z e^{-k_{n-1}\frac{2\pi\ii}{d_n}\frac{|z|-s_{n-1}}{r_{n-1}-s_{n-1}} t})  & \text{if } z\in D_{n-1}', \\
F_1^{n-2}(z e^{-k_{n-1}\frac{2\pi\ii}{d_n} t})  & \text{if } z\in A_n'\cup D_\infty', \\
F_1^{n-2}(z)  & \text{other}.
\end{cases}
\end{equation}
Then $F_t^{n-1}:\EC\to\EC$ is quasiregular, holomorphic in $\EC\setminus \bigcup_{i=1}^{n-1}\overline{D}_i'$ for every $t\in[0,1]$, $F_t^{n-1}$ depends continuously on $t\in[0,1]$ and $\partial F_t^{n-1}(z)/\partial z$, $\partial F_t^{n-1}(z)/\partial \overline{z}$ depend continuously on $t\in[0,1]$ for every $z\in\EC\setminus\bigcup_{j=1}^{n-1}\partial D_j'$. Moreover,
\begin{itemize}
\item $F_0^{n-1}=F_1^{n-2}$;
\item $F_1^{n-1}=\widetilde{F}$ on $\EC\setminus \bigcup_{j=1}^{n-1}D_j'$; and
\item $F_1^{n-1}|_{\overline{D}_j'}$ is homotopic to $F|_{\overline{D}_j'}$ rel $\partial D_j'$ for all $1\leq j\leq n-1$.
\end{itemize}

For $t\in [0,1]$, we define
\begin{equation}
F_t(z):=
\begin{cases}
\widetilde{F}(z)  & \text{if } t=0, \\
F_{(n-1)t-(i-1)}^i  & \text{if } t\in(\frac{i-1}{n-1},\frac{i}{n-1}] \text{ for }1\leq i\leq n-1.
\end{cases}
\end{equation}
Then $F_t:\EC\to\EC$ is quasiregular, holomorphic in $\EC\setminus \bigcup_{i=1}^{n-1}\overline{D}_i'$ for every $t\in[0,1]$, $F_t$ depends continuously on $t\in[0,1]$ and $\partial F_t(z)/\partial z$, $\partial F_t(z)/\partial \overline{z}$ depend continuously on $t\in[0,1]$ for every $z\in\EC\setminus\bigcup_{j=1}^{n-1}\partial D_j'$. Moreover,
\begin{itemize}
\item $F_1=F_1^{n-1}=\widetilde{F}$ on $\EC\setminus \bigcup_{j=1}^{n-1} D_j'$; and
\item $F_1|_{\overline{D}_j'}$ is homotopic to $F|_{\overline{D}_j'}$ rel $\partial D_j'$ for all $1\leq j\leq n-1$.
\end{itemize}

Finally, let $(\widehat{F}_t:\EC\to\EC)_{t\in[0,1]}$ be a continuous path of quasiregular maps such that
$\widehat{F}_0=F_1$, $\widehat{F}_1=F$ and $\widehat{F}_t=F_1$ on $\EC\setminus \bigcup_{i=1}^{n-1} D_i'$.
In particular, the path can be chosen such that\footnote{The reason is that both $F_1$ and $F$ are holomorphic in $\EC\setminus \bigcup_{i=1}^{n-1}\overline{D}_i'$ and $C^1$-continuous in $\bigcup_{i=1}^{n-1}D_i'$. See Lemma \ref{lem:homotopy-Cui}.} $\widehat{F}_t$ is holomorphic in $\EC\setminus \bigcup_{i=1}^{n-1}\overline{D}_i'$ for every $t\in[0,1]$, $\widehat{F}_t$ depends continuously on $t\in[0,1]$ and $\partial \widehat{F}_t(z)/\partial z$, $\partial \widehat{F}_t(z)/\partial \overline{z}$ depend continuously on $t\in[0,1]$ for every $z\in\EC\setminus\bigcup_{j=1}^{n-1}\partial D_j'$.
Denote by
\begin{equation}
\widetilde{F}_t(z):=
\begin{cases}
F_{2t}(z)  & \text{if } t\in[0,1/2], \\
\widehat{F}_{2t-1}(z)  & \text{if } t\in (1/2,1].
\end{cases}
\end{equation}
Then $\widetilde{F}_t:\EC\to\EC$ is quasiregular, holomorphic in $\EC\setminus \bigcup_{i=1}^{n-1}\overline{D}_i'$ for every $t\in[0,1]$, $\widetilde{F}_t$ depends continuously on $t\in[0,1]$, and $\partial \widetilde{F}_t(z)/\partial z$, $\partial \widetilde{F}_t(z)/\partial \overline{z}$ depend continuously on $t\in[0,1]$ for every $z\in\EC\setminus\bigcup_{j=1}^{n-1}\partial D_j'$. Moreover,
\begin{itemize}
\item $\widetilde{F}_0=\widetilde{F}$, $\widetilde{F}_1=F$; and
\item $\widetilde{F}_1=\widetilde{F}$ on $\EC\setminus \bigcup_{i=1}^{n-1} D_i'$.
\end{itemize}
Therefore, although $\widetilde{F}$ and $F$ are (probably) different quasiregular extensions of the IFS $\mathscr{L}(\MC)$ on $\bigcup_{i=1}^n A_i'$, we have found a continuous path of quasiregular maps $(\widetilde{F}_t:\EC\to\EC)_{t\in[0,1]}$ connecting $\widetilde{F}$ with $F$.


\medskip
\textbf{Step 4}: \textit{The continuous paths in the hyperbolic component}.
Let $\sigma_0$ be the standard conformal structure on $\EC$ represented by the zero Beltrami differential. For each $t\in[0,1]$ we define a measure conformal structure function
\begin{equation}
\sigma_t(z):=
\begin{cases}
\sigma_0(z)  & \text{if } z\in D_0'\cup D_\infty', \\
((\widetilde{F}_t^{\circ \ell})^*\sigma_0)(z)  & \text{if } z\in \widetilde{F}_t^{-(\ell-1)}(\bigcup_{i=1}^{n-1}D_i') \text{ for some } \ell\geq 1, \\
\sigma_0(z)  & \text{other}.
\end{cases}
\end{equation}
Since each $\widetilde{F}_t$ is holomorphic in $\EC\setminus\bigcup_{i=1}^{n-1}\overline{D}_i'$, it is easy to see that $\sigma_t$ has bounded dilatation and is invariant under the action of $\widetilde{F}_t$. According to Measurable Riemann Mapping Theorem, there exists a unique quasiconformal map $\xi_t:\EC\to\EC$ which solves the Beltrami equation $\xi_t^*(\sigma_0)=\sigma_t$ and fixes $0$, $1$ and $\infty$. Note that $\sigma_t$ depends continuously on $t\in[0,1]$ (since each $\widetilde{F}_t$ is holomorphic in $\EC\setminus \bigcup_{i=1}^{n-1}\overline{D}_i'$ and $\partial \widetilde{F}_t(z)/\partial z$, $\partial \widetilde{F}_t(z)/\partial \overline{z}$ depend continuously on $t\in[0,1]$ for every $z\in\EC\setminus\bigcup_{j=1}^{n-1}\partial D_j'$).
By Ahlfors-Bers theorem \cite{AB60}, the map
\begin{equation}
\widetilde{f}_t:=\xi_t\circ \widetilde{F}_t\circ \xi_t^{-1}
\end{equation}
is a rational map which depends continuously on $t\in[0,1]$. In particular, $\widetilde{f}_0=\xi_0\circ \widetilde{F}\circ \xi_0^{-1}$, $\widetilde{f}_1=\xi_1\circ F\circ \xi_1^{-1}$ and each $\widetilde{f}_t$ with $t\in[0,1]$ is a hyperbolic rational map with a Cantor circle Julia set.

Since $\widetilde{F}=\phi_\infty\circ f\circ\phi_\infty^{-1}$ and $\widetilde{f}_0=\xi_0\circ \widetilde{F}\circ \xi_0^{-1}$, we have
\begin{equation}
\widetilde{f}_0=\widetilde{\phi}\circ f\circ\widetilde{\phi}^{-1},
\end{equation}
where $\widetilde{\phi}:=\xi_0\circ\phi_\infty:\EC\to\EC$ is a quasiconformal mapping.
For $s\in[0,1]$, define a conformal structure $\widetilde{\sigma}_s=s\widetilde{\phi}^*(\sigma_0)$. Since $f$ is a rational map, $\widetilde{\sigma}_s$ is preserved by $f$. By the Measurable Riemann
Mapping Theorem, there exists a unique quasiconformal mapping $\zeta_s:\EC\to\EC$ which solves the Beltrami equation $\zeta_s^*(\sigma_0)=\widetilde{\sigma}_s$ and fixes $0$, $1$ and $\infty$. Define $\widehat{f}_s:=\zeta_s\circ f\circ\zeta_s^{-1}$. Then $\widehat{f}_s$ is a hyperbolic rational map with a Cantor circle Julia set for all $s\in[0,1]$. According to Ahlfors-Bers \cite{AB60}, $(\widehat{f}_s)_{s\in[0,1]}$ is a continuous path connecting $f$ with $\widetilde{f}_0$.

For $t\in[0,1]$, we define
\begin{equation}
f_t(z):=
\begin{cases}
\widehat{f}_{2t}(z)  & \text{if } t\in[0,1/2], \\
\widetilde{f}_{2t-1}(z)  & \text{if } t\in (1/2,1].
\end{cases}
\end{equation}
Then $f_t$ depends continuously on $t\in[0,1]$ and each $f_t$ is a hyperbolic rational map with a Cantor circle Julia set. In particular, $(f_t)_{t\in[0,1]}$ is a continuous path in the hyperbolic component $\MH$ connecting $f_0=f$ with $f_1=\xi_1\circ F\circ \xi_1^{-1}$.

\medskip
\textbf{Step 5}: \textit{The conclusion}.
If we begin with the rational map $g$ whose combination is also $\MC=(\Io;d_1,\cdots,d_n)$, then as above one can find a continuous path $(g_t)_{t\in[0,1]}$ in a hyperbolic component connecting $g_0=g$ with $g_1=\xi_1\circ F\circ \xi_1^{-1}$ (Careful: not $\xi_1\circ G\circ \xi_1^{-1}$ for some $G$ since $F$ is the given quasiregular mapping depending only on the combination $\MC$, see \eqref{equ:F-qr}). Therefore, $f$ and $g$ can be connected by a continuous path in the hyperbolic component $\MH$.
This completes the proof of Theorem \ref{thm:conj-resta} and hence Theorem \ref{thm:conj}.
\end{proof}

\begin{rmk}
Let $f_1$ and $f_2$ be two hyperbolic rational maps with degree $d$ whose Julia sets are Cantor circles. If $f_1$ and $f_2$ have the same combinatorial data in $\mathscr{C}$, then from Theorem \ref{thm:conj-resta} we know that they lie in the same hyperbolic component of the moduli space $\MM_d$.
\end{rmk}

Recall that $f_{\varrho,d_1,\cdots,d_n}$ is the family introduced in \eqref{family-QYY}.
Let
\begin{equation}
d_{\max}:=\max\{d_1,\cdots,d_n\} \quad\text{and}\quad \eta:=\sum_{i=1}^n\frac{1}{d_i}<1.
\end{equation}
The parameters $a_1,\cdots,a_{n-1}$ in Theorem \ref{thm-QYY} can be chosen more specifically as in the following theorem (see \cite[Theorem 2.5]{QYY15}).

\begin{thm}\label{parameter-restate}
Let $u_1=\tau\,d_{\max}^{-5}$, $v_1=\tau\,d_{\max}^{-2}$;  and $u_0=\tau^{1+1/d_n+2(1-\eta)/3}$, $v_0=\tau^{1/d_n+(1-\eta)/3}$.
\begin{enumerate}
\item For $\varrho=1$, set $|a_{n-1}|=v_1^{1/d_n}$ and $|a_i|=u_1^{1/d_{i+1}}|a_{i+1}|$ for $1\leq i\leq n-2$;
\item For $\varrho=0$, set $|a_{n-1}|=v_0^{1/d_n}$ and $|a_i|=u_0^{1/d_{i+1}}|a_{i+1}|$ for $1\leq i\leq n-2$.
\end{enumerate}
Then $J(f_{\varrho,d_1,\cdots,d_n})$ is a Cantor circle if $\tau>0$ is small enough.
\end{thm}

If $\tau>0$ is small enough, we have
\begin{equation}\label{equ:a-1-a-2}
0<|a_1|\ll |a_2|\ll \cdots\ll |a_{n-1}|\ll 1.
\end{equation}
Since at least one of $0$ and $\infty$ (or both) lies in the super-attracting basins of $f_{\varrho,d_1,\cdots,d_n}$, we can define the corresponding annulus $A_i$ with $1\leq i\leq n$ and $D_i$ with $1\leq i\leq n-1$ for $f_{\varrho,d_1,\cdots,d_n}$ (see \S\ref{subsec-comb}). From \cite[Lemma 2.4]{QYY15} we know that $D_i$ contains the circle $\T_{|a_i|}$ and $d_i+d_{i+1}$ critical points for all $1\leq i\leq n-1$.

\medskip
In the following, we always assume that $a_i's$ are chosen as in Theorem \ref{parameter-restate} such that the Julia set of $f_{\varrho,d_1,\cdots,d_n}$ is a Cantor set of circles.
Then there are following four cases (Here we denote by $f:=f_{\varrho,d_1,\cdots,d_n}$ for simplicity):
\begin{enumerate}
\item If $\varrho=1$ and $n$ is even, then $f(D_0)=D_\infty$ and $f(D_\infty)=D_\infty$;
\item If $\varrho=1$ and $n$ is odd, then $f(D_0)=D_0$ and $f(D_\infty)=D_\infty$;
\item If $\varrho=0$ and $n$ is odd, then $f(D_0)=D_\infty$ and $f(D_\infty)=D_0$;
\item If $\varrho=0$ and $n$ is even, then $f(D_0)=D_0$ and $f(D_\infty)=D_0$.
\end{enumerate}

Note that up to topological conjugacies, we only need to consider the first three cases since every map of case (d) is conjugate to some map of case (a) on their corresponding Julia sets (compare \S\ref{subsec-comb}). In particular, cases (a), (b) and (c) have combinations $(\Io;d_1,\cdots,d_n)$, $(\It;d_1,\cdots,d_n)$ and $(\Is;d_1,\cdots,d_n)$ respectively.
As an immediate corollary of Theorem \ref{thm:conj-resta}, we have

\begin{cor}\label{cor:typical-ele}
Any Cantor circle hyperbolic component $\MH$ in $\MM_d$ contains at least one map $f_{\varrho,d_1,\cdots, d_n}$ with the parameters $a_1,\cdots,a_{n-1}$ given in Theorem \ref{parameter-restate}.
\end{cor}

If a hyperbolic component of rational maps of degree $d\geq 2$ has compact closure in $\MM_d$, then this hyperbolic component is called \emph{bounded}. A theorem of Makienko asserts that if the Julia set of a hyperbolic rational map is disconnected, then the hyperbolic component containing this rational map is \textit{unbounded} (see \cite{Mak00}). Note that each Cantor circle Julia set is disconnected. Therefore, we have

\begin{cor}\label{cor:unbounded}
All Cantor circle hyperbolic components in $\MM_d$ are unbounded.
\end{cor}

Based on Corollary \ref{cor:typical-ele}, we can give another proof of Corollary \ref{cor:unbounded} by avoiding the use of Makienko's theorem.

\begin{proof}[{Another proof of Corollary \ref{cor:unbounded}}]
Let $\MH$ be a Cantor circle hyperbolic component in $\MM_d$. By Corollary \ref{cor:typical-ele}, $\MH$ contains at least one map $f_{\varrho,d_1,\cdots,d_n}$ in \eqref{family-QYY}. In particular, the parameters $a_1$, $\cdots$, $a_{n-1}$ of $f_{\varrho,d_1,\cdots,d_n}$ can be chosen arbitrarily small (see Theorem \ref{parameter-restate}). Therefore, $f_{\varrho,d_1,\cdots,d_n}$ is a small perturbation of $z\mapsto z^{d_1}$ or $z\mapsto z^{-d_1}$. It implies that $\MH$ is unbounded since $\deg(f_{\varrho,d_1,\cdots,d_n})=\sum_{i=1}^n d_i>d_1$ for $n\geq 2$.
\end{proof}

\section{Number of Cantor circle hyperbolic components}\label{sec-No-CC}

The aim of this section is to calculate the number of Cantor circle hyperbolic components in the moduli space $\MM_d$, for any given $d\geq 2$.

\begin{prop}\label{no-comp-intro}
Let $f$ be a rational map whose Julia set is a Cantor circle. Then $\deg(f)\geq 5$.
\end{prop}

\begin{proof}
If $d\leq 4$, then \eqref{equ:combi-data} has no solution.
\end{proof}

Note that Proposition \ref{no-comp-intro} is also valid for parabolic rational maps.
In the following we use $\sharp \,X$ to denote the cardinal number of a finite set $X$.

\begin{thm}\label{no-of-hyper-com}
For every $d\geq 5$, the number $N(d)$ of Cantor circle hyperbolic components in $\MM_d$ is calculated by \eqref{equ:N-d}.
\end{thm}

\begin{proof}
According to Theorem \ref{thm:conj}, it is sufficient to calculate the different topologically conjugate classes of the rational maps when they restrict on the Cantor circle Julia sets. For this, we consider the combinations of such rational maps. There are three types in all (see \S\ref{subsec-comb}). Obviously, the dynamics on these three types of Cantor circle Julia sets are not topologically conjugate to each other.

\medskip
For each given $d\geq 5$, we define
\begin{equation}
N^\Io:= \left\{(\Io;d_1,\cdots,d_n)\left|~\sum_{i=1}^nd_i=d, \sum_{i=1}^n\frac{1}{d_i}<1~\text{and}~ n\geq 2 \text{~is even}\right.\right\}.
\end{equation}
For each given $d\geq 5$ and $\kappa\in\{\It,\Is\}$, we define
\begin{equation}
\begin{split}
N_1^\kappa&:=
\left\{(\kappa;d_1,\cdots,d_n)
\left|
\begin{array}{l}
\sum_{i=1}^nd_i=d, \sum_{i=1}^n\frac{1}{d_i}<1 \\
(d_1,\cdots,d_n) = (d_n,\cdots,d_1) ~\text{and}~ n\geq 3 \text{~is odd}
\end{array}
\right.
\right\} \text{ and} \\
N_2^\kappa&:=
\left\{(\kappa;d_1,\cdots,d_n)
\left|
\begin{array}{l}
\sum_{i=1}^nd_i=d, \sum_{i=1}^n\frac{1}{d_i}<1 \\
(d_1,\cdots,d_n)\neq (d_n,\cdots,d_1) ~\text{and}~ n\geq 3 \text{~is odd}
\end{array}
\right.
\right\}.
\end{split}
\end{equation}

Note that $\sharp N_1^\It=\sharp N_1^\Is$ and $\sharp N_2^\It=\sharp N_2^\Is$. By Lemma \ref{lema:diff-comb}, the number of different topologically conjugate classes (consider the restriction on the Julia sets) of the rational maps lying in Cantor circle hyperbolic components in $\MM_d$ is calculated by
\begin{equation}
N(d)=\,\sharp N^\Io + \sharp N_1^\It + \sharp N_2^\It/2 + \sharp N_1^\Is + \sharp N_2^\Is/2
=\, (\sharp N^\Io+\sharp N_1^\It+\sharp N_2^\It) +\sharp N_1^\It.
\end{equation}
This ends the proof of Theorem \ref{no-of-hyper-com} and hence Theorem \ref{thm:number-CC}.
\end{proof}

By the enumerative method, one can calculate $N(d)$ easily for any given $d\geq 5$ by Theorem \ref{thm:number-CC}. See Table \ref{Tab_number-of-new}.

\begin{table}[htpb]
\renewcommand{\arraystretch}{1.1}
\begin{center}
\begin{tabular}{|c|c|c|c|c|c|c|c|c|}\hline
\rowcolor{mygreen}
 $d$     & 5     & 6     & 7      & 8      & 9      & 10     & 11     & 12     \\ \hline
 $N(d)$  & 2     & 3     & 4      & 5      & 6      & 11     & 22     & 37     \\ \hline\hline
\rowcolor{mygreen}
 $d$     & 13    & 14    & 15     & 16     & 17     & 18     & 19     & 20     \\ \hline
 $N(d)$  & 46    & 57    & 68     & 81     & 110    & 159    & 228    & 290    \\ \hline\hline
\rowcolor{mygreen}
 $d$     & 21    & 22    & 23     & 24     & 25     & 26     & 27     & 28     \\ \hline
 $N(d)$  & 410   & 519   & 716    & 872    & 1070   & 1323   & 1722   & 2258   \\  \hline\hline
\rowcolor{mygreen}
 $d$     & 29    & 30    & 31     & 32     & 33     & 34     & 35     & 36     \\ \hline
 $N(d)$  & 3066  & 4227  & 5566   & 6950   & 8604   & 10483  & 12916  & 15838 \\ \hline
\end{tabular}
\vskip0.1cm
 \caption{The list of the number $N(d)$ of Cantor circle hyperbolic components in the moduli space of rational maps of degree $d$, where $5\leq d\leq 36$. }
 \label{Tab_number-of-new}
\end{center}
\end{table}


\section{Hausdorff dimension of Cantor circles: The infimum}\label{sec:dim-inf}

Recall that $f_{\varrho,d_1,\cdots,d_n}$ is the family defined in Theorem \ref{family-QYY}. The aim of this section is to find the infimum of the Hausdorff dimensions of the Julia sets of the rational maps in the Cantor circle hyperbolic components. Since the lower bound of the Hausdorff dimensions of the Cantor circle Julia sets can be obtained easily (see Proposition \ref{prop:conf-dim-new}), according to Corollary \ref{cor:typical-ele}, it is sufficient to work with the family $f_{\varrho,d_1,\cdots,d_n}$ and prove that it can produce a sequence of Hausdorff dimensions which approach the lower bound. Then the lower bound becomes the infimum.

\subsection{Conformal dimension of Cantor circle Julia sets}\label{subsec-conf-dim}

Let $X$ be a metric space. The \textit{conformal dimension} $\dim_{C}(X)$ of $X$ is the infimum of the Hausdorff dimensions of all metric spaces which are quasisymmetrically equivalent to $X$. Note that the conformal dimension is an invariant of the quasisymmetrically equivalent class of a metric space.
Recall that $\alpha_{d_1,\cdots,d_n}\in(0,1)$ is the number determined by Equation \eqref{equ:root}.

\begin{prop}\label{prop:conf-dim-new}
Let $\MH$ be a Cantor circle hyperbolic component whose combination is $\MC=(\kappa;d_1,\cdots,d_n)\in\mathscr{C}$. Then $\dim_C(J(f))=1+\alpha_{d_1,\cdots,d_n}$ for all $f\in\MH$.
\end{prop}

\begin{proof}
According to Theorem \ref{thm:qs-classifi}, the Julia set of each $f\in\MH$ is quasisymmetrically equivalent to a standard Cantor circle $J(\MC)$. To prove this proposition we use the following fact (see \cite[Proposition 2.9]{Pan89} or \cite[Proposition 3.7]{Hai09}): if $X$ is a $\lambda$-Ahlfors regular metric space, then $X\times[0,1]$ equipped with the product metric has conformal dimension $1+\lambda$.
Note that the standard Cantor set $A(\MC)$ is an $\alpha$-Ahlfors regular metric space with $\alpha=\alpha_{d_1,\cdots,d_n}$ (see Lemma \ref{lema:Hdim}). Hence the conformal dimension of the Julia set of $f$ is $1+\alpha$.
\end{proof}

Let $J_{\varrho,d_1,\cdots,d_n}$ be the Julia set of $f_{\varrho,d_1,\cdots,d_n}$ for $n\geq 2$.
The following result is an immediate consequence of Proposition \ref{prop:conf-dim-new} and Corollary \ref{cor:typical-ele}.

\begin{cor}\label{cor:conf-dim}
The conformal dimension of $J_{\varrho,d_1,\cdots,d_n}$ is $1+\alpha_{d_1,\cdots,d_n}$.
\end{cor}

\begin{rmk}
If $d_i=d_0>n$ for all $1\leq i\leq n$, then $\sum_{i=1}^n 1/d_i=n/d_0<1$ and
\begin{equation}
\dim_C(J_{\varrho,d_1,\cdots,d_n})=1+\frac{\log n}{\log d_0}.
\end{equation}
\end{rmk}

\subsection{Falconer's criterion and its extension to conformal IFS}\label{subsec:Falconer}

In this subsection, we first introduce Falconer's criterion to calculate the upper bounds of the Hausdorff dimensions of the attractors. Then we develop a criterion to calculate the lower bound of the Hausdorff dimension of the attractor.

The following criterion is useful for calculating the upper bound of the Hausdorff dimension of the attractor of an IFS (see \cite[Proposition 9.6]{Fal14}).

\begin{thm}\label{th:Fal-upper}
Let $\mathscr{F}=\{\psi_1,\ldots,\psi_m\}$ be an IFS on the closed set $\Omega\subset \mathbb{R}^n$ satisfying $|\psi_i(x)-\psi_i(y)|\leq c_i|x-y|$, where $0<c_i<1$ and $i=1,\ldots, m$. Then the Hausdorff dimension of the attractor $J$ of $\mathscr{F}$ satisfies: $\dim_H (J)\leq s$, where $s>0$ is the unique number satisfying
\begin{equation}
\quad \sum_{i=1}^m c_i^{s}=1.
\end{equation}
\end{thm}

For a hyperbolic rational function $f$ with degree $d\geq 2$, $f$ is strictly expanding in a neighborhood of $J(f)$ because the critical orbit is far from the Julia set. In some cases, $f^{-1}$ can be defined and has $d$ inverse branches $g_1$, $\cdots$, $g_d$ which form an IFS whose attractor is exactly the Julia set of $f$.
Therefore, Theorem \ref{th:Fal-upper} can be used to calculate the upper bound of the Hausdorff dimension of the Julia sets of some hyperbolic rational maps.

\medskip
The IFS $\lbrace \psi_1, \ldots, \psi_m \rbrace$ is said to satisfy the \textit{open set condition}, if there is a non-empty bounded open set $V$, such that $\forall\,1\leq i\leq m$, $\psi_i$ is defined on $\overline{V}$ and $\bigsqcup_{i=1}^m \psi_i (V)\subset V$, where ``$\bigsqcup$" denotes disjoint union. Note that $V$ may not contain $J$, but $\overline{V}\supset J$ (see \cite[Theorem 9.1] {Fal14}).

In \cite[Proposition 9.7]{Fal14}, Falconer gave a similar statement to Theorem \ref{th:Fal-upper} to calculate the lower bound of the Hausdorff dimension of the attractor of an IFS. But in the statement an additional condition that the IFS $\mathscr{F}$ should satisfy the ``\textit{strong open set condition}" was added. Under this condition, the attractor of $\mathscr{F}$ must be a Cantor set. So the result of Falconer cannot be used to deal with the case when the Julia sets are not Cantor sets. To overcome this difficulty, we introduce the concept of conformal IFS.

\begin{defi}[conformal IFS]
Let $\mathscr{F}=\{\psi_1,\ldots,\psi_m\}$ be an IFS on the closed set $\Omega\subset\C$. We call that $\mathscr{F}$ is a \textit{conformal IFS}, if there is an open neighborhood $W$ of $\Omega$ and a family of univalent functions $\widetilde{\mathscr{F}}=\{\widetilde {\psi}_1,\ldots,\widetilde{\psi}_m: W\to\C\}$, such that $\widetilde{\psi}_i(W)\subset W$ and $\widetilde{\psi}_i|_{\Omega}=\psi_i$, where $i\in\{1,\cdots,m\}$.
\end{defi}

\begin{rmk}
The definition of conformal IFS here is different from \cite{MU96} and \cite{MU99}, where the conformal IFS has no relation to holomorphic maps in $\C$.
\end{rmk}

We need to use the following distortion theorem on univalent functions (see \cite[Theorem 1.6]{Pom75}).

\begin{thm}[{Koebe distortion theorem}]\label{thm:Koebe}
Let $f:\D\to\C$ be a univalent function satisfying $f(0)=0$ and $f'(0)=1$. For any $z\in\D $, we have
\begin{enumerate}
\item $\tfrac{|z|}{(1+|z|)^2}\leq|f(z)|\leq\tfrac{|z|}{(1-|z|)^2}$ ; and
\item $\tfrac{1-|z|}{(1+|z|)^3}\leq|f'(z)|\leq\tfrac{1+|z|}{(1-|z| )^3}$.
\end{enumerate}
\end{thm}

We use the following criterion to calculate the lower bound of the Hausdorff dimension of the Julia sets of rational functions\footnote{In the applications of Theorem \ref{th:Fal}, sometimes it is necessary to make a coordinate transformation of the dynamical plane. Otherwise one cannot define conformal IFS. For example, in \S\ref{subsec:IFS}, we need to make a logarithmic transformation.} and the proof is inspired by \cite[Theorem 9.3]{Fal14}. Note that the result of \cite[Theorem 9.3]{Fal14} can be only applied to similarities. That means the contracting ratio of the mappings in the IFS is the same at every point. For conformal IFS, although the contracting ratios are different in different places, we can still use Theorem \ref{thm:Koebe} to control the distortion.

\begin{thm}\label{th:Fal}
Let $\mathscr{F}=\{\psi_1,\ldots,\psi_m\}$ be a \emph{conformal} IFS on the closed set $\Omega\subset \C$ satisfying the open set condition. If $|\psi_i(x)-\psi_i(y)|\geq b_i|x-y|$, where $0<b_i<1$ and $i=1,\ldots, m$, then the Hausdorff dimension of the attractor $J$ of $\mathscr{F}$ satisfies: $\dim_H (J)\geq s$, where $s>0$ is the unique number satisfying
\begin{equation}
\quad \sum_{i=1}^m b_i^{s}=1.
\end{equation}
\end{thm}

\begin{proof}
Let $\MI=\{(i_1,i_2,\cdots):1\leq i_j\leq m \text{ and } j\geq 1\}$ be an index set containing the infinite sequences, and $\MI^k=\{(i_1,\cdots,i_k):1\leq i_j\leq m, 1\leq j\leq k\}$ is an index set containing the finite sequences. For a given finite sequence $(i_1,\cdots,i_k)\in\MI^k$, we denote
 $\MI_{i_1,\cdots,i_k}=\{(i_1,\cdots,i_k,q_{k+1},q_{k+2},\cdots):1\leq q_{k+j} \leq m \text{ and } j\geq 1\}$, and $J_{i_1,\cdots, i_k}:=\psi_{i_1}\circ\cdots\circ\psi_{i_k}(J )$. Let
\begin{equation}\label{equ:mu-i}
\mu(\MI_{i_1,\cdots,i_k})=(b_{i_1}\cdots b_{i_k})^s.
\end{equation}
According to $\mu(\MI_{i_1,\cdots,i_k})=\sum_{i=1}^m\mu(\MI_{i_1,\cdots,i_k,i})$, $\mu$ is a mass distribution on $\MI$. It also induces a mass distribution $\widetilde{\mu}$ on the attractor $J$, which is defined as
\begin{equation}\label{equ:mu-tilde}
\widetilde{\mu}(A)=\mu\,\{(i_1,i_2,\cdots):x_{i_1,i_2,\cdots}\in A\}, \quad \forall\, A\subset J,
\end{equation}
where $x_{i_1,i_2,\cdots}=\bigcap_{k=1}^\infty J_{i_1,\cdots, i_k}$ and $\widetilde{\mu}(J)=1$.

Because $\mathscr{F}$ satisfies the open set condition, there is an open set $V$ such that
\begin{equation}
\bigsqcup_{i=1}^m \psi_i (V)\subset V
\text{\quad and \quad} \bigcup_{i=1}^m \psi_i (\overline{V})\subset \overline{V}.
\end{equation}
From \cite[Theorem 9.1]{Fal14} we have $J\subset \overline{V}$. For any $(i_1,\cdots,i_k)\in\MI^k$, we denote $V_{i_1,\cdots, i_k}:=\psi_{i_1}\circ\cdots\circ\psi_{i_k}(V)$. Then $J_{i_1,\cdots, i_k}\subset \overline{V}_{i_1,\cdots, i_k}$.
For any small disk $B$ with radius $r>0$, we will consider those $V_{i_1,\cdots, i_k}$'s whose diameters are comparable to $r$ and whose closures intersect with $J\cap B$, to estimate $\widetilde{\mu}(B)$.

Since $\mathscr{F}$ is a conformal IFS, there is an open set $W$ satisfying $W\supset \Omega\supset J$ and $\forall\,1\leq i\leq m$, the map $\psi_i:\Omega\to\Omega$ can be extended to a univalent function $\widetilde{\psi}_i:W\to W$ on $W$. Without loss of generality we assume that $\overline{V}\subset W$. Otherwise one can use some $k$-th image of $V$:
\begin{equation}
V^k:=\bigcup_{(i_1,\cdots,i_k)\in\MI^k}V_{i_1,\cdots, i_k}
\end{equation}
to replace $V$. Note that $V^k$ also satisfies the open set condition. Since the elements in $\mathscr{F}$ are uniformly contracting maps, there are constants $C>0$ and $0 <\eta<1$ such that for any $k\geq 1$, and $(i_1,\cdots,i_k)\in\MI^k$,
\begin{equation}\label{equ:unif-contra}
\diam(V_{i_1,\cdots,i_k})< C\eta^k.
\end{equation}
Denote $\delta:=\dist(V,\partial W)=\inf\{|z_1-z_2|:z_1\in V, z_2\in\partial W\}>0$. Then there is $k_0\geq 1$ such that $\forall\,(i_1,\cdots,i_{k_0})\in\MI^{k_0}$, we have
\begin{equation}
\diam(V_{i_1,\cdots,i_{k_0}})< \delta/2.
\end{equation}
So there is a constant $a_1$, $a_2>0$ such that $\forall\,(i_1,\cdots,i_{k_0})\in\MI^{k_0}$, there is $y_{i_1,\cdots, i_{k_0}}\in V_{i_1,\cdots,i_{k_0}}$ satisfying
\begin{equation}\label{equ:nested}
B_{a_1}(y_{i_1,\cdots,i_{k_0}})\subset V_{i_1,\cdots,i_{k_0}}\subset B_{a_2}(y_{i_1,\cdots,i_{k_0} })
\subset B_{a_2+\delta/2}(y_{i_1,\cdots,i_{k_0}})\subset W.
\end{equation}

In the following we assume that
\begin{equation}
0<r\leq \min_{(i_1,\cdots,i_{k_0})\in\MI^{k_0}}\diam (V_{i_1,\cdots,i_{k_0}}).
\end{equation}
For any given infinite sequence $(i_1,i_2,\cdots)\in\MI$, there must exist a minimal $k\geq k_0$ such that
\begin{equation}\label{equ:r-dom}
\Big(\min_{1\leq i\leq m} b_i\Big)r \leq \diam (V_{i_1,\cdots,i_k}) \leq r.
\end{equation}
Let $\MQ$ be the collection of all such finite sequences $(i_1,\cdots,i_k)$. By \eqref{equ:unif-contra}, $\MQ$ is a finite set. Since $V_1$, $\cdots$, $ V_m$ are pairwise disjoint, so $V_{i_1,\cdots,i_k,1}$, $\cdots$, $V_{i_1,\cdots,i_k,m}$ are also disjoint. This implies that the elements in the family of open sets $\{V_{i_1,\cdots,i_k}:(i_1,\cdots,i_k)\in\MQ\}$ are disjoint, and
\begin{equation}
J\subset \bigcup_{(i_1,\cdots,i_k)\in\MQ} J_{i_1,\cdots,i_k}\subset \bigcup_{(i_1,\cdots,i_k)\in\MQ} \overline{V }_{i_1,\cdots,i_k}.
\end{equation}

According to \eqref{equ:nested} and Theorem \ref{thm:Koebe}, there are constants $C_1$, $C_2>0$ such that for any $(i_1,\cdots,i_k)\in\MQ$, the open set $V_{i_1,\cdots,i_k}$ contains a disk with radius $C_1\,\diam (V_{i_1,\cdots,i_k})$, and $V_{i_1,\cdots,i_k}$ is contained in a disk with radius $C_2\,\diam (V_{i_1,\cdots,i_k})$. By \eqref{equ:r-dom}, $V_{i_1,\cdots,i_k}$ contains a disk with radius $C_1\, \big(\min_{1\leq i\leq m} b_i\big)r$, and is contained in a disk with radius $C_2\,r$.

Let $\MQ_1=\{(i_1,\cdots,i_k)\in\MQ:B\cap \overline{V}_{i_1,\cdots,i_k}\neq\emptyset\}$. By \cite [Lemma 9.2]{Fal14}, the number of the elements in $\MQ_1$ satisfies
\begin{equation}
|\MQ_1|\leq M:=\frac{(1+2C_2)^2}{C_1^2\,\big(\min_{1\leq i\leq m} b_i\big)^2}.
\end{equation}
If $x_{i_1,i_2,\cdots}\in J\cap B\subset\bigcup_{(j_1,\cdots,j_k)\in\MQ_1} \overline{V}_{j_1,\cdots,j_k}$, then there is $k\geq k_0$ such that $(i_1,\cdots, i_k)\in \MQ_1$. Combining \eqref{equ:mu-i}, \eqref{equ:mu-tilde} and \eqref{equ:r-dom}, we have
\begin{equation}
\begin{split}
\widetilde{\mu}(B)
=&~\widetilde{\mu}(B\cap J)=\mu\{(i_1,i_2,\cdots):x_{i_1,i_2,\cdots}\in B\cap J\} \\
\leq&~\mu\Big(\bigcup_{(i_1,\cdots,i_k)\in\MQ_1} \MI_{i_1,\cdots,i_k}\Big)=\sum_{(i_1,\cdots,i_k)\in\MQ_1}\mu (\MI_{i_1,\cdots,i_k})\\
=&~\sum_{(i_1,\cdots,i_k)\in\MQ_1}(b_{i_1}\cdots b_{i_k})^s\leq \sum_{(i_1,\cdots,i_k)\in\MQ_1}(\diam (V_{i_1,\cdots,i_k}))^s \\
\leq& \sum_{(i_1,\cdots,i_k)\in\MQ_1}r^s\leq r^s M.
\end{split}
\end{equation}
Since any set $E\subset\C$ is contained in a disk with radius $\diam (E)$, we have $\widetilde{\mu}(E)\leq (\diam(E))^s M$. By mass distribution principle (see \cite[p.\,67]{Fal14}), the $s$-dimension Hausdorff measure of $J$ is at least $1/M$. This implies that $\dim_H(J)\geq s$.
\end{proof}

\subsection{Decomposition of the dynamical planes}

We have calculated the conformal dimension of $J_{\varrho,d_1,\cdots,d_n}$ in \S\ref{subsec-conf-dim}. To compute the Hausdorff dimension of $J_{\varrho,d_1,\cdots,d_n}$, we need to decompose the dynamical planes and estimate the expanding factor near the Julia sets.
In the rest of this section, we assume that the parameters $a_1$, $\cdots$, $a_{n-1}$ are \textit{positive} numbers evaluated as in Theorem \ref{parameter-restate}.

For small $\alpha>0$, $\tau>0$ and every $1\leq i\leq n-1$, we define the following numbers:
\begin{equation}\label{equ:const-R}
\begin{split}
R_0=R_0^+=\tau, \quad
& R_i^-=\tau^\alpha a_i, \quad \text{and}\quad\\
R_i^+=\tau^{-\alpha}a_i, \quad
& R_\infty=R_n^-=(2/\tau)^{1/d_n}.
\end{split}
\end{equation}

Recall that the disks $D_0$, $D_\infty$ and the annuli $D_i$ with $1\leq i\leq n-1$, $A_i$ with $1\leq i\leq n$ are defined for\footnote{When $\varrho$, $d_1$, $\cdots$, $d_n$ are given, the parameters $a_1$, $\cdots$, $a_{n-1}$ are functions of variable $\tau>0$.}
\begin{equation}
f_\tau:=f_{\varrho,d_1,\cdots,d_n}.
\end{equation}
For $0<r_1<r_2<\infty$, recall that $\A(r_1,r_2):=\{z\in\C:r_1<|z|<r_2\}$. The following result\footnote{In \cite[Lemma 2.4]{QYY15}, $A_i$ is an annulus containing the critical circle. But in this paper we use $A_i$ to denote the annulus between every two adjacent critical annuli.} has been included in the proof of \cite[Lemma 2.4]{QYY15}.

\begin{lem}\label{lema:annulus}
There exists a small $\alpha>0$ such that if $\tau> 0$ is small enough, then
\begin{equation}
\begin{split}
& \overline{\D}_{R_0}\subset D_0, \quad \overline{\A}(R_i^-,R_i^+)\subset D_i \text{\quad with } 1\leq i\leq n-1, \text{ and}\\
&  \EC\setminus \D_{R_\infty}\subset D_\infty, \quad A_i\subset \A(R_{i-1}^+,R_i^-) \text{\quad with } 1\leq i\leq n.
\end{split}
\end{equation}
Moreover, $f_\tau(\A(R_{i-1}^+,R_i^-))$ contains $\overline{\A}(R_0,R_\infty)$, where $1\leq i\leq n$. All the critical values of $f_\tau$ are contained in $\D_{R_0}\cup (\EC\setminus\overline{\D}_{R_\infty})$.
\end{lem}

\subsection{Logarithmic coordinates and the proof}\label{subsec:IFS}

Note that $f_\tau$ is a real rational map and $f_\tau^{-1}(\overline{\A}(R_0,R_\infty)\setminus\R^+)$ consists of $d_i$ components in $\A(R_{i-1}^+,R_i^-)\setminus\R^+$ for every $1\leq i\leq n$. We label the closure of these $d_i$ components of $f_\tau^{-1}(\overline{\A}(R_0,R_\infty)\setminus\R^+)$ in $\A(R_{i-1}^+,R_i^-)\setminus\R^+$ counterclockwise as $S_{i,1}$, $S_{i,2}$, $\cdots$, $S_{i,d_i}$, such that $S_{i,1}$ lies above of $\R^+$, $S_{i,d_i}$ lies below of $\R^+$, $S_{i,1}\cap\R^+\neq\emptyset$ and $S_{i,d_i}\cap\R^+\neq\emptyset$.

\medskip
Let $\Xi_i=\{(i,\ell):\ell=1,2,\cdots,d_i\}$ for $1\leq i\leq n$ and
\begin{equation}
\Xi:=\Xi_1\cup\Xi_2\cup\cdots\cup\Xi_n
\end{equation}
be the index sets.
We denote $S:=\overline{\A(R_0,R_\infty)\setminus\R^+}$ and treat every $z\in[R_0,R_\infty]$ as two different points in $S$, i.e., $S$ is seen as a simply connected closed domain.

Based on the convenience introduced above, one can see that $f_\tau|_{S_\xi}:S_\xi\to S$ is a homeomorphism for every $\xi\in\Xi$. Let $\varphi_\xi:S\to S_\xi$ be the inverse of $f_\tau|_{S_\xi}$. Then every $\varphi_\xi$ is a contracting mapping and $\{\varphi_\xi:\xi\in\Xi\}$ forms an iterated function system. The attractor of $\{\varphi_\xi:\xi\in\Xi\}$ is exactly the Julia set $J_\tau$ of $f_\tau$.

\medskip
To compute the Hausdorff dimension of $J_\tau$, we need Theorems \ref{th:Fal-upper} and \ref{th:Fal}.
From Lemma \ref{lema:annulus} one can see that the IFS $\{\varphi_\xi:\xi\in\Xi\}$ satisfies the open condition. In order to estimate the contracting constants, we lift the map $f_\tau$ (and the IFS) to the logarithmic coordinate.

Note that $S:=\overline{\A(R_0,R_\infty)\setminus \R^+}$ is seen as a simply connected closed domain. We lift $S$ and $S_\xi$ under
\begin{equation}
\sigma:Z\mapsto z=e^Z
\end{equation}
to obtain $\widetilde{S}$ and $\widetilde{S}_\xi$ such that $\widetilde{S}=\{Z:0\leq \im Z\leq 2\pi \text{ and } e^Z\in S\}$ and $\widetilde{S}_\xi=\{Z:0\leq \im Z\leq 2\pi \text{ and } e^Z\in S_\xi\}$. This lift is unique determined and $\widetilde{S}_{i,1}\cap\R^+\neq\emptyset$ for all $1\leq i\leq n$.
For every $\xi\in\Xi$, we define $F_{\xi,\tau}(Z)$ on $\widetilde{S}_\xi$ by
\begin{equation}\label{equ:F-epsilon}
\begin{split}
F_{\xi,\tau}(Z):=&~\sigma^{-1}\circ f_\tau\circ\sigma(Z)=\log f_\tau(e^Z)\\
=&~(-1)^{n-\varrho}\left(d_1 Z+\sum_{i=1}^{n-1}(-1)^i\log(e^{(d_i+d_{i+1})Z}-a_i^{d_i+d_{i+1}})\right).
\end{split}
\end{equation}
Here $\sigma^{-1}$ is a continuous branch of the logarithm which maps $S$ and $\widetilde{S}$.
Then $F_{\xi,\tau}$ is a homeomorphism from $\widetilde{S}_\xi$ to $\widetilde{S}$.

Let $\Phi_\xi:\widetilde{S}\to\widetilde{S}_\xi$ be the inverse of $F_{\xi,\tau}$. Then $\{\Phi_\xi:\xi\in\Xi\}$ forms an IFS defined on $\widetilde{S}$ which is conjugated by $\log$ to the IFS $\{\varphi_\xi:\xi\in\Xi\}$. The attractor of $\{\Phi_\xi:\xi\in\Xi\}$ is $\widetilde{J}_\tau:=\{Z:0\leq \im Z\leq 2\pi \text{ and } e^Z\in J_\tau\}$. Hence we have
\begin{equation}\label{equ:Haus-equl}
\dim_H(J_\tau)=\dim_H(\widetilde{J}_\tau).
\end{equation}

\begin{proof}[{Proof of the first part of Theorem \ref{thm:hdim-CC}}]
We first estimate the asymptotic behavior of $F_{\xi,\tau}'(Z)$ as $\tau\to 0$. By \eqref{equ:F-epsilon} we have
\begin{equation}\label{equ:F-sigma-d}
F_{\xi,\tau}'(Z)=(-1)^{n-\varrho}\left(d_1+\sum_{i=1}^{n-1}(-1)^i\frac{(d_i+d_{i+1})z^{d_i+d_{i+1}}}{z^{d_i+d_{i+1}}-a_i^{d_i+d_{i+1}}}\right),
\end{equation}
where $Z\in \widetilde{S}_\xi$ and $z=e^Z\in S_\xi$.

By Lemma \ref{lema:annulus}, if $\tau>0$ is sufficiently small, we have
\begin{equation}
\begin{split}
0<R_0=R_0^+&~\ll R_1^-\ll a_1 \ll R_1^+\ll \cdots\\
&~\ll R_{n-1}^-\ll a_{n-1}\ll R_{n-1}^+\ll 1\ll R_n^-=R_\infty
\end{split}
\end{equation}
and
\begin{equation}\label{equ:limit-a}
\lim_{\tau\to 0}\frac{R_i^-}{a_i}=\lim_{\tau\to 0}\frac{a_i}{R_i^+}=0, \text{ where }1\leq i\leq n-1.
\end{equation}
If $\xi=(i,\ell)$ with $1\leq i\leq n$ and $1\leq \ell\leq d_i$, then $z\in S_\xi$ implies
\begin{equation}
R_{i-1}^+<|z|<R_i^-.
\end{equation}
Therefore, by \eqref{equ:F-sigma-d} and \eqref{equ:limit-a}, if $z\in S_{i,\ell}$ we have
\begin{equation}
\widetilde{c}_i(\tau)\leq \Big|\big|F_{\xi,\tau}'(Z)|-\big|d_1+\sum_{j=1}^{i-1}(-1)^j (d_j+d_{j+1})\big|\Big|=\Big||F_{\xi,\tau}'(z)|-d_i\Big|\leq \widetilde{b}_i(\tau),
\end{equation}
where $\widetilde{b}_i(\tau)$ and $\widetilde{c}_i(\tau)$ are positive numbers depending on $\tau$ (also on $z$) which satisfy
\begin{equation}\label{equ:limit-b-c}
\lim_{\tau\to 0} \widetilde{b}_i(\tau)=\lim_{\tau\to 0} \widetilde{c}_i(\tau)=0
\end{equation}
uniformly on $S_{i,\ell}$, where $1\leq i\leq n$ and $1\leq \ell\leq d_i$. Hence if $\xi=(i,\ell)$ with $1\leq\ell\leq d_i$ we have
\begin{equation}
\widehat{c}_i(\tau)\leq |F_{\xi,\tau}'(Z)|\leq \widehat{b}_i(\tau),
\end{equation}
where
\begin{equation}
\widehat{b}_i(\tau):=d_i+\widetilde{b}_i(\tau) \quad \text{and}\quad \widehat{c}_i(\tau):=d_i-\widetilde{c}_i(\tau).
\end{equation}

Note that each $\Phi_\xi$ can be extended to be a univalent function defined in a neighborhood of $\widetilde{S}$ and $\Phi_\xi(\widetilde{S})= \widetilde{S}_\xi$. It follows that $\{\Phi_\xi:\xi\in\Xi\}$ forms a conformal IFS defined in a neighborhood of $\widetilde{S}$ and satisfies the open set condition.

Set $b_i=1/\widehat{b}_i(\tau)$ and $c_i=1/\widehat{c}_i(\tau)$, where $1\leq i\leq n$. By \eqref{equ:Haus-equl}, Theorems \ref{th:Fal-upper} and \ref{th:Fal}, we have
\begin{equation}
\beta_-\leq\dim_H(\widetilde{J}_\tau)=\dim_H(J_\tau)\leq \beta_+,
\end{equation}
where $\beta_-$ and $\beta_+$ satisfy
\begin{equation}
\sum_{i=1}^n d_i b_i^{\beta_-}=1 \quad\text{and}\quad \sum_{i=1}^n d_i c_i^{\beta_+}=1.
\end{equation}

By the definition of $b_i$ and $c_i$, we have
\begin{equation}
\sum_{i=1}^n \frac{d_i}{(d_i+\widetilde{b}_i(\tau))^{\beta_-}}=1 \quad\text{and}\quad
\sum_{i=1}^n \frac{d_i}{(d_i-\widetilde{c}_i(\tau))^{\beta_+}}=1.
\end{equation}
Let $\alpha_{d_1,\cdots,d_n}\in(0,1)$ be the number determined by \eqref{equ:root}. From the above equations, we see that
\begin{equation}
\lim_{\tau\to 0}\beta_-=1+\alpha_{d_1,\cdots,d_n}=\lim_{\tau\to 0}\beta_+.
\end{equation}
This implies that $\lim\limits_{\tau\to 0}\dim_H(J_\tau)=1+\alpha_{d_1,\cdots,d_n}$. Combining Proposition \ref{prop:conf-dim-new}, this completes the proof of the first assertion of Theorem \ref{thm:hdim-CC}.
\end{proof}

\section{Hausdorff dimension of Cantor circles: The supremum}\label{sec:dim-sup}

In this section we study the supremum of the Hausdorff dimensions of the Cantor circle Julia sets. The idea is to perturb some parabolic rational maps with Cantor circle Julia sets to the hyperbolic ones and then use Shishikura's result about parabolic bifurcations. In fact, we will prove the second part of Theorem \ref{thm:hdim-CC} for more general hyperbolic components.

\medskip
The following theorem is a weak version of \cite[Theorem 2]{Shi98}.

\begin{thm}[{Shishikura}]\label{thm-shishi}
Suppose that a rational map $f_0$ of degree $d\geq 2$ has a parabolic fixed point $z_0$ with multiplier $1$ and that the immediate parabolic basin of $z_0$ contains only one critical point of $f_0$. Then for any $\varepsilon>0$ and $b>0$, there exist a neighborhood $\MN$ of $f_0$ in the space of rational maps of degree $d$, a neighborhood $V$ of $z_0$ in $\EC$, positive integers $N_1$ and $N_2$ such that if $f\in\MN$, and if $f$ has a fixed point in $V$ with multiplier $e^{2\pi\ii\alpha}$, where
\begin{equation}\label{equ:Shishi-multi}
\alpha=\pm\frac{1}{a_1\pm\frac{1}{a_2+\beta}}
\end{equation}
with integers $a_1\geq N_1$, $a_2\geq N_2$ and $\beta\in\C$, $0\leq\re\,\beta<1$, $|\im\beta|\leq b$, then
\begin{equation}
\dim_H(J(f))>2-\varepsilon.
\end{equation}
\end{thm}

For the shape of the region for $\alpha$ satisfying \eqref{equ:Shishi-multi}, see \cite[Figure 3]{Shi98}.

\begin{thm}\label{thm-Hdim-sup}
Let $\MH$ be a hyperbolic component in $\MM_d$ with $d\geq 2$. Suppose that every $f\in\MH$ has a simply connected periodic Fatou component whose closure is disjoint with any other Fatou components. Then
\begin{equation}
\sup_{f\in\MH}\dim_H (J(f))=2.
\end{equation}
\end{thm}

\begin{proof}
By the assumption, every $f\in\MH$ has a cycle of attracting periodic Fatou components $U_0\to U_1\to\cdots\to U_{p-1}\to U_0$ which are all simply connected, where $p\geq 1$. Moreover, $\overline{U}_i\cap \overline{U}_j=\emptyset$ for any $i\neq j$. By performing a quasiconformal surgery, it is easy to see that $\MH$ contains at least one map $f_0$ such that $U_0\to U_1\to\cdots\to U_{p-1}\to U_0$ is a cycle of super-attracting basins of $f_0$ and $f_0^{\circ p}:U_0\to U_0$ contains exactly one critical point $0$ (counted without multiplicity). By a standard quasiconformal surgery \cite[Chap.\,4]{BF14}, one can construct a continuous path $(f_t:\EC\to\EC)_{t\in[0,1)}$ of hyperbolic rational maps in $\MH$ such that $f_t^{\circ p}$ has a geometrically attracting fixed point $0$ with multiplier $t$ whose immediate attracting basin $U_0^t$ contains exactly one critical point (counted without multiplicity).

According to \cite{CT18}, $(f_t:\EC\to\EC)_{t\in[0,1)}$ can be chosen as a pinching path and the limit $f_1:=\lim_{t\to1^-}f_t$ exists, where $f_1$ is a parabolic rational map having the following properties:  
\begin{enumerate}
\item $f_1^{\circ p}$ has a parabolic fixed point at $0$ with multiplier $1$ whose immediate parabolic basin $U_0^1$ contains exactly one critical point\footnote{Note that $f_1^{\circ p}$ has exactly one petal (contained in $U_0^1$) at the parabolic fixed point $0$. This is the reason why we assumed that each $f\in\MH$ has a simply connected periodic Fatou component whose closure is disjoint with any other Fatou components.};
\item $f_1^{\circ p}|_{J(f_1)}$ is topologically conjugate to $f_t^{\circ p}|_{J(f_t)}$ for all $t\in[0,1)$ (actually topologically conjugate to $f^{\circ p}|_{J(f)}$ for all $f\in\MH$); and
\item The Julia set of $f_1$ is homeomorphic to the Julia set of $f_0$.
\end{enumerate}

By Theorem \ref{thm-shishi}, for any $\varepsilon>0$, there exist a small neighborhood $\MN_\varepsilon$ of $f_1^{\circ p}$ in the moduli space $\MM_{d'}$ with $d'=d^p$ and a subset $\MN_\varepsilon'\subset\MN_\varepsilon$, such that every $f\in\MN_\varepsilon'\cap\MH$ has a cycle of geometrically attracting periodic point with multiplier satisfying\footnote{One can perturb $f_1$ along \textit{horocycles} to obtain the required multipliers, see \cite[\S 12]{McM00a}.} \eqref{equ:Shishi-multi}, and the Hausdorff dimension of $J(f)$ is at least $2-\varepsilon$. Therefore we have $\sup_{f\in\MH}\dim_H (J(f))=2$.

If $\MH$ is a Cantor circle hyperbolic component, then the closures of any two different Fatou components of $f\in\MH$ are disjoint. Moreover, every $f\in\MH$ has a cycle of simply connected periodic Fatou components. This ends the proof of Theorem \ref{thm-Hdim-sup} and the second part of Theorem \ref{thm:hdim-CC}.
\end{proof}

\begin{rmk}
(i) The Fatou components in Theorem \ref{thm-Hdim-sup} could be infinitely connected. Indeed, the maps in this theorem are only required to contain one simply connected attracting basin but the other attracting basins may be infinitely connected.

\medskip
(ii) Theorem \ref{thm-Hdim-sup} can be used to study the Hausdorff dimension of some other kind of Julia sets. For example, for any Sierpi\'{n}ski carpet hyperbolic component $\MH$ (i.e., every map in $\MH$ has a Sierpi\'{n}ski carpet Julia set), one has $\sup_{f\in\MH}\dim_H (J(f))=2$ (see \cite{FY20}).
\end{rmk}

\begin{proof}[{Proof of Theorem \ref{thm:Hdim-sharp}}]
Let $\MH$ be a Cantor circle hyperbolic component in $\MM_{2d}$ such that each $f\in\MH$ has the combination $(\Io;d,d)\in\mathscr{C}$, where $d\geq 3$. By Theorem \ref{thm:hdim-CC}, we have
\begin{equation}
\inf_{f\in\MH}\dim_H (J(f))=1+\frac{\log 2}{\log d} \text{\quad and\quad}\sup_{f\in\MH}\dim_H (J(f))=2.
\end{equation}
Note that $f\mapsto \dim_H (J(f))$ is a continuous function as $f$ moves in $\MH$ (see \cite{Rue82}). For each $s\in(1+\log 2/\log d, 2)$, there exists a map $f\in\MH$ such that $\dim_H (J(f))=s$. Since $d$ can be chosen arbitrarily large, the second statement of Theorem \ref{thm:Hdim-sharp} follows.

Let $f$ be a rational map with a Cantor circle Julia set $J(f)$. According to \cite{QYY15}, $f$ is hyperbolic or parabolic. By \cite{Urb94} or \cite{Yin00}, we have $\dim_H(J(f))<2$. If $f$ is hyperbolic, then $f$ is contained in some Cantor circle hyperbolic component and we have $\dim_H(J(f))\geq \dim_C(J(f))>1$ by Proposition \ref{prop:conf-dim-new}. Suppose that $f$ is parabolic. By the continuity of the Hausdorff dimension of the Julia sets (see \cite[Theorem 11.2]{McM00a}), there exists a sequence of hyperbolic rational maps $f_n$ such that $\dim_H(J(f))=\lim_{n\to\infty}\dim_H(J(f_n))$, where $\{f_n\}_{n\in\N}$ are contained in the same Cantor circle hyperbolic component. This means that $\dim_H(J(f))\geq \inf_{n\in\N}\dim_H(J(f_n))\geq \dim_C(J(f_n))>1$. Therefore, we have $1<\dim_H(J(f))<2$ if $J(f)$ is a Cantor set of circles.
\end{proof}




\appendix
\section{Homotopic classes from annuli to disks}\label{sec:appendix}

It is not difficult to show that all the topological branched covering maps from a Jordan disk to another Jordan disk with the same boundary values are in the same homotopic class.
In this section we focus our attention on classifying the homotopic classes of such maps defined from an annulus to a Jordan disk.


Recall that $\A_r=\{z\in\C:r<|z|<1\}$ and $\T_r=\{z\in\C:|z|=r\}$, where $0<r<1$. In particular, $\T=\T_1$ is the unit circle.
Let $p$, $q\geq 1$ be two integers. We denote $\omega_j=e^{2\pi\ii\frac{(j-1)q}{p+q}}$ for $1\leq j\leq p+q$.

\begin{lem}[{see Figure \ref{Fig:cpreimagea}}]\label{lem:twist}
Let $f:\overline{\A}_r\to \overline{\D}$ be a continuous map satisfying
\begin{itemize}
\item $f:\A_r\to \D$ is a branched covering map with degree $p+q$;
\item $\deg(f|_{\T_r})=p$, $\deg(f|_{\T})=q$; and
\item $f$ has $p+q$ different critical points $CP=\{c_j:1\leq j\leq p+q\}$ in $\A_r$, and $p+q$ different critical values $CV=\{v_j=f(c_j):1\leq j\leq p+q\}$ in $\D$.
\end{itemize}
Then for any given $b_1\in f^{-1}(1)\cap\T_r$, there are $p+q$ smooth arcs $\{\gamma_j:1\leq j\leq p+q\}$ such that
\begin{enumerate}
\item $\gamma_j$ connects $v_j$ with $\omega_j$;
\item $\gamma_j\setminus\{\omega_j\}\subset\D$ and $\gamma_j\cap\gamma_k=\emptyset$ for any $1\leq j\neq k\leq p+q$;
\item the connected component of $f^{-1}(\gamma_1)$ containing $c_1$ passes through $b_1$; and
\item for any connected component $U'$ of $f^{-1}(\D\setminus\bigcup_{j=1}^{p+q}\gamma_j)$,
\begin{equation}
\big(f^{-1}(\cup_{j=1}^{p+q}\gamma_j\cup\T)\cap\overline{U'}\big)\setminus\T_r \text{\quad and\quad}
\big(f^{-1}(\cup_{j=1}^{p+q}\gamma_j\cup\T)\cap\overline{U'}\big)\setminus\T
\end{equation}
have $p$ and $q$ connected components respectively.
\end{enumerate}
\end{lem}

\begin{figure}[!htpb]
  \setlength{\unitlength}{1mm}
  \centering
  \includegraphics[width=0.9\textwidth]{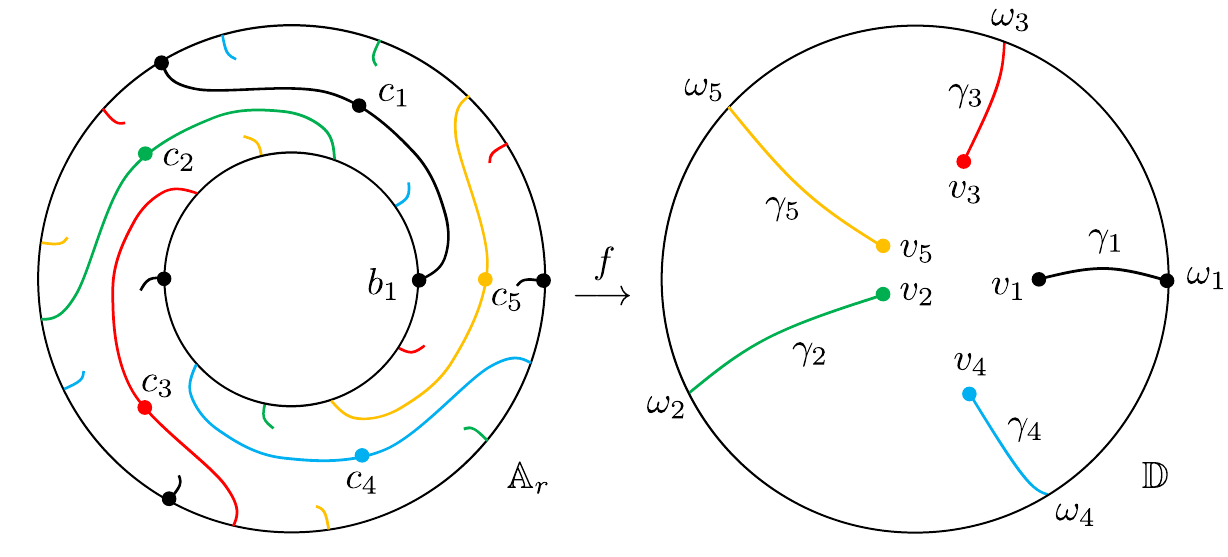}
  \caption{The $p+q$ smooth arcs $\gamma_1$, $\cdots$, $\gamma_{p+q}$ and their preimages under the branched covering map $f:\A_r\to\D$, where $p=2$ and $q=3$.}
  \label{Fig:cpreimagea}
\end{figure}

\begin{proof}
Let $\eta_j$ be a simple curve (including two end points) connecting $v_j$ with $1$ for $1\leq j\leq p+q$ such that $\eta_j\cap\eta_k=\{1\}$ for $1\leq j\neq k\leq p+q$ and $\eta_j\setminus\{1\}\subset\D$ for all $1\leq j\leq p+q$. Then $W:=\D\setminus\bigcup_{j=1}^{p+q}\eta_j$ is a simply connected domain and $f^{-1}(W)$ consists of $p+q$ simply connected domain $W_1$, $\cdots$, $W_{p+q}$. Moreover, $f:W_j\to W$ is a homeomorphism for all $1\leq j\leq p+q$.

We claim that for every $1\leq j\leq p+q$, the connected component of $f^{-1}(\eta_j)$ containing $c_j$ is a simple curve connecting a point in $f^{-1}(1)\cap\T_r$ with a point in $f^{-1}(1)\cap\T$. Otherwise, it is easy to verify that the number of connected components of $f^{-1}(W)$ would be less than $p+q$, which is a contradiction. For any given $b_1\in f^{-1}(1)\cap\T_r$, there exists $1\leq j=j(b_1)\leq p+q$, such that the connected component of $f^{-1}(\eta_j)$ containing $c_j$ is a simple curve connecting $b_1$ with a point in $f^{-1}(1)\cap\T$. Otherwise, there exists a unique connected component $W_j$ of $f^{-1}(W)$ whose boundary containing $b_1$ for which the restriction of $f$ on $W_j$ has degree at least two, which is a contradiction. Without loss of generality (by permutating the subscripts if necessary), we assume that $j=1$.

Let $\gamma_1:=\eta_1$. We define $p+q-1$ smooth arcs $\{\gamma_j:2\leq j\leq p+q\}$ such that every $\gamma_j$ connects $v_j$ with $\omega_j$ and they satisfy $\gamma_j\setminus\{\omega_j\}\subset\D$ and $\gamma_j\cap\gamma_k=\emptyset$ for any $1\leq j\neq k\leq p+q$. If $\{\gamma_j:1\leq j\leq p+q\}$ satisfy the statement (d), then the proof is finished. In the following, we assume that there exists at least one connected component $U'$ of $f^{-1}(\D\setminus\bigcup_{j=1}^{p+q}\gamma_j)$ which does not satisfy (d), see Figure \ref{Fig:cpreimageb}. In the following we adjust the positions of the curves $\{\gamma_j:2\leq j\leq p+q\}$ and exchange the subscripts such that statement (d) holds.

\begin{figure}[!htpb]
  \setlength{\unitlength}{1mm}
  \centering
  \includegraphics[width=0.9\textwidth]{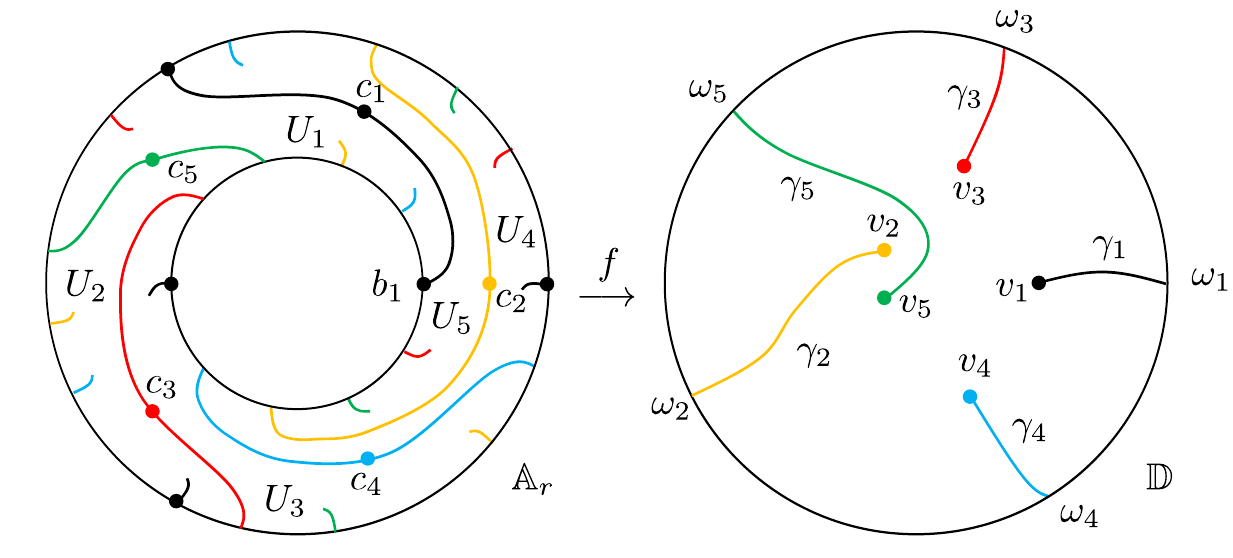}
  \caption{The candidates of $p+q$ smooth arcs $\gamma_1$, $\cdots$, $\gamma_{p+q}$ and their preimages under the branched covering map $f:\A_r\to\D$, where $p=2$ and $q=3$. As a connected component of $f^{-1}(\D\setminus\bigcup_{j=1}^{p+q}\gamma_j)$, $U_1$ does not satisfy (d). Compare Figure \ref{Fig:cpreimagea}.}
  \label{Fig:cpreimageb}
\end{figure}

For $1\leq j\leq p+q$, let $\beta_j$ be the connected component of $f^{-1}(\gamma_j)$ containing $c_j$. Note that $f^{-1}(\D\setminus\bigcup_{j=1}^{p+q}\gamma_j)$ consists of $p+q$ connected components. We label them by $U_1$, $U_2$, $\cdots$, $U_{p+q}$ anticlockwise, where $U_1$ is the component lying on the left of $\beta_1$ (recall that one end point of $\beta_1$ is $b_1$). For $1\leq j\leq p+q$, $\overline{U}_j\setminus(\T_r\cup\T)$ consists of $p+q$ connected components, whose closures are
\begin{equation}\label{equ:seq}
\beta_{j_1}, \alpha_{j_2},\cdots, \alpha_{j_{k(j)}}, \beta_{j_{k(j)+1}}, \alpha_{j_{k(j)+2}}, \cdots, \alpha_{j_{p+q}},
\end{equation}
where $f:\beta_{j_1}\to\gamma_{j_1}$ and $f:\beta_{j_{k(j)+1}}\to\gamma_{j_{k(j)+1}}$ are two-to-one, $f:\alpha_{j_\ell}\to\gamma_{j_\ell}$ is a homeomorphism for all $1\leq \ell\leq p+q$ ($\ell\neq 1$, $k(j)+1$) and $\bigcup_{\ell=1}^{p+q}\gamma_{j_\ell}=\bigcup_{j=1}^{p+q}\gamma_j$. Here the sequence \eqref{equ:seq} is labelled such that one end point of $\alpha_{j_\ell}$ attaches on $\T$ for $2\leq \ell\leq k(j)$ while one end point of $\alpha_{j_\ell}$ attaches on $\T_r$ for $k(j)+2\leq \ell\leq p+q$. Moreover, the curves in \eqref{equ:seq} are listed by the same order as $\{\omega_j:1\leq j\leq p+q\}$ on $\T$ which is induced by the homeomorpshim $f:U_j\to \D\setminus\bigcup_{\ell=1}^{p+q}\gamma_\ell$. This implies that $\big(f^{-1}(\cup_{\ell=1}^{p+q}\gamma_\ell\cup\T)\cap\overline{U}_j\big)\setminus\T_r$ and $\big(f^{-1}(\cup_{\ell=1}^{p+q}\gamma_\ell\cup\T)\cap\overline{U}_j\big)\setminus\T$, respectively, have $p+q-k(j)$ and $k(j)$ connected components.

\medskip
To guarantee the statement (d), we begin with the smallest $j\in[1,p+q]\cap\mathbb{N}$ for which $U'=U_j$ does not satisfy (d). Then $k(j)\neq q$ (see Figure \ref{Fig:cpreimageb} for the case $k(1)=2<q=3$). We replace the old critical value curves $\gamma_{j_{k(j)+1}}$ and $\gamma_{j_{q+1}}$ by a pair of new ones $\widetilde{\gamma}_{j_{k(j)+1}}$ and $\widetilde{\gamma}_{j_{q+1}}$ respectively, where $\widetilde{\gamma}_{j_{k(j)+1}}$ is a smooth arc connecting $v_{j_{q+1}}$ with $\omega_{j_{k(j)+1}}$ and $\widetilde{\gamma}_{j_{q+1}}$ is a smooth arc connecting $v_{j_{k(j)+1}}$ with $\omega_{j_{q+1}}$. Moreover, these two arcs are chosen such that $\widetilde{\gamma}_{j_{k(j)+1}}\setminus\{\omega_{j_{k(j)+1}}\}\subset\D$, $\widetilde{\gamma}_{j_{q+1}}\setminus\{\omega_{j_{q+1}}\}\subset\D$ and they are disjoint with each other and disjoint with other $\gamma_\ell$ for $\ell\neq j_{k(j)+1}$, $j_{q+1}$. Then we exchange the subscripts of $v_{j_{k(j)+1}}$ and $v_{j_{q+1}}$, and denote the new curves $\gamma_{j_{k(j)+1}}:=\widetilde{\gamma}_{j_{k(j)+1}}$ and $\gamma_{j_{q+1}}:=\widetilde{\gamma}_{j_{q+1}}$. Now we have a new set of critical value curves $\{\gamma_j:1\leq j\leq p+q\}$. One can define new $c_j$, $\beta_j$, $\alpha_{j_\ell}$ and $U_j$ etc, similarly as in the previous paragraph. Note that exchanging the subscripts of the old $v_{j_{k(j)+1}}$ and $v_{j_{q+1}}$ does not effect the validity of the statement (d) for $U'=U_\ell$ with $1\leq \ell\leq j-1$. This implies that for this new critical value curves $\{\gamma_j:1\leq j\leq p+q\}$, statement (d) holds for the new $U_\ell$, where $1\leq \ell\leq j$.

\medskip
If statement (d) holds for the new $U_\ell$ for all $j+1\leq \ell\leq p+q$, then the proof is finished. Otherwise, let $j'\in[j+1,p+q]$ be the smallest integer for which $U'=U_{j'}$ does not satisfy (d). Then we have a similar sequence as \eqref{equ:seq} and $k(j')\neq q$. Similar to the previous argument, we replace the previous critical value curves $\gamma_{j'_{k(j')+1}}$ and $\gamma_{j'_{q+1}}$ by a pair of newer critical value curves $\widetilde{\gamma}_{j'_{k(j')+1}}$ and $\widetilde{\gamma}_{j'_{q+1}}$ respectively, where $\widetilde{\gamma}_{j'_{k(j')+1}}$ is a smooth arc connecting $v_{j'_{q+1}}$ with $\omega_{j'_{k(j')+1}}$ and $\widetilde{\gamma}_{j'_{q+1}}$ is a smooth arc connecting $v_{j'_{k(j')+1}}$ with $\omega_{j'_{q+1}}$. Moreover, these two arcs are chosen such that $\widetilde{\gamma}_{j'_{k(j')+1}}\setminus\{\omega_{j'_{k(j')+1}}\}\subset\D$, $\widetilde{\gamma}_{j'_{q+1}}\setminus\{\omega_{j'_{q+1}}\}\subset\D$ and they are disjoint with each other and disjoint with other $\gamma_\ell$ for $\ell\neq j'_{k(j')+1}$, $j'_{q+1}$. We exchange the subscripts of $v_{j'_{k(j')+1}}$ and $v_{j'_{q+1}}$, and denote $\gamma_{j'_{k(j')+1}}:=\widetilde{\gamma}_{j'_{k(j')+1}}$ and $\gamma_{j'_{q+1}}:=\widetilde{\gamma}_{j'_{q+1}}$. Then we have a newer critical value curves $\{\gamma_j:1\leq j\leq p+q\}$. One can define newer $c_j$, $\beta_j$, $\alpha_{j_\ell}$ and $U_j$ etc, similarly as before. Moreover, for this newer critical value curves $\{\gamma_j:1\leq j\leq p+q\}$, statement (d) holds for the newer $U_\ell$, where $1\leq \ell\leq j'$.

\medskip
Inductively, after finite steps, one can obtain a brand new critical value curves $\{\gamma_j:1\leq j\leq p+q\}$ such that they not only satisfy (a)-(c), but also satisfy (d) for the corresponding new  $U'=U_j$, where $1\leq j\leq p+q$.
\end{proof}

For $q\geq 1$, we define a \textit{partial-twist} map $T_q:\overline{\A}_r\to\overline{\A}_r$ as:
\begin{equation}
T_q(z):=z e^{\frac{2\pi\ii}{q}\frac{|z|-r}{1-r}},
\end{equation}
where $z\in\overline{\A}_r$. Note that $T_q$ fixes the inner boundary of $\A_r$. It is easy to see that
$T_q^{\circ q}(z)=z e^{ 2\pi\ii\frac{|z|-r}{1-r}}$
is a \textit{full-twist} of $\overline{\A}_r$. See Figure \ref{Fig:partial-twist}.

\begin{figure}[!htpb]
  \setlength{\unitlength}{1mm}
  \centering
  \includegraphics[width=0.75\textwidth]{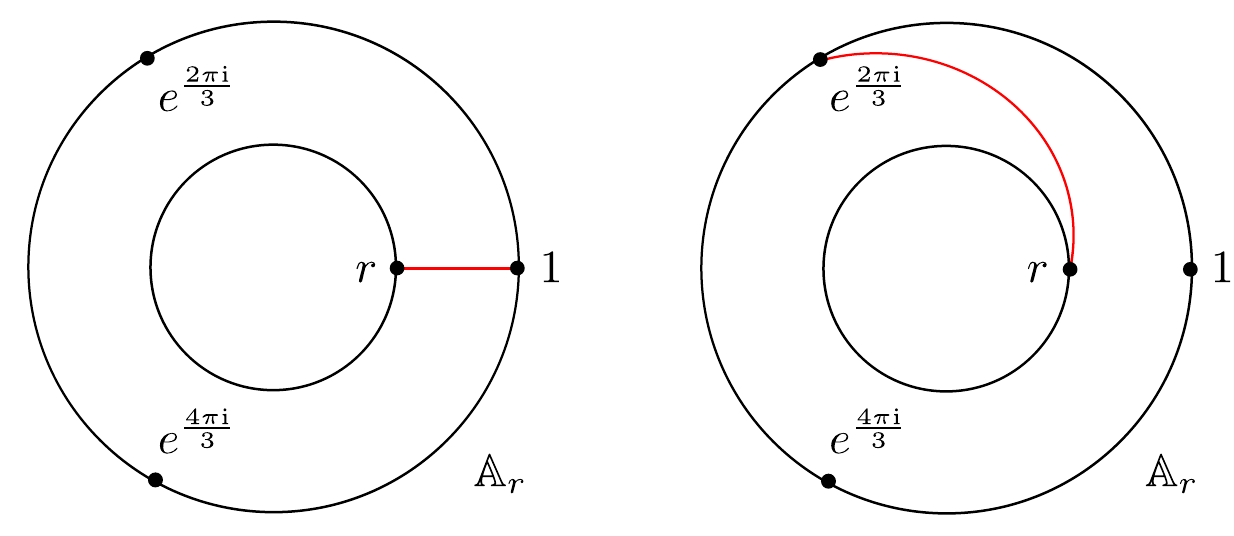}
  \caption{The segment $[r,1]$ and its partial-twist under $T_q$ with $q=3$.}
  \label{Fig:partial-twist}
\end{figure}


\begin{lem}\label{lem:homotopy-Cui}
Suppose that $f_0$, $f_1:\overline{\A}_r\to \overline{\D}$ are continuous maps satisfying
\begin{itemize}
\item $f_0$, $f_1:\A_r\to \D$ are both branched covering maps of degree $p+q$;
\item $f_0(z)=f_1(z)=z^q$ for $z\in\T$ and $f_0(z)=f_1(z)=r^p/z^p$ for $z\in\T_r$.
\end{itemize}
Then there exist an integer $k$ and a continuous map
\begin{equation}
\Phi:[0,1]\times \overline{\A}_r\to \overline{\D}
\end{equation}
such that
\begin{itemize}
\item $\Phi(0,\cdot)=f_0\circ T_q ^{\circ k }$ and $\Phi(1,\cdot)=f_1$;
\item $\Phi(t,\cdot)|_{\partial\A_r}=f_0|_{\partial\A_r}$ for all $t\in[0,1]$; and
\item $\forall\, t\in[0,1]$, $\Phi(t,\cdot):\overline{\A}_r\to \overline{\D}$ is a branched covering of degree $p+q$.
\end{itemize}
Further, if both $f_0$ and $f_1$ are $C^1$-quasiregular, then $\Phi(t,\cdot):\A_r\to\D$ can be chosen such that it is $C^1$-quasiregular for all $t\in[0,1]$ and $\partial\Phi(t,z)/\partial z$, $\partial\Phi(t,z)/\partial \overline{z}$ are continuous for $(t,z)\in[0,1]\times\A_r$.
\end{lem}

\begin{proof}
According to Riemann-Hurwitz's formula, $f_0$ (resp. $f_1$) has $d$ critical points in $\A_r$ and $d$ critical values in $\D$ (counted with multiplicity).
Without loss of generality, we assume that the $d$ critical values are different. In particular, the corresponding critical points are all simple (i.e., with local degree two). Otherwise, one can make a small continuous perturbation on $f_0$ (resp. $f_1$).

In the following we call the set of critical value curves $\{\gamma_j:1\leq j\leq p+q\}$ in Lemma \ref{lem:twist} \textit{admissible}. Note that
\begin{equation}\label{equ:f0-1}
f_0^{-1}(1)=\{e^{2\pi\ii\frac{j}{q}}:1\leq j\leq q\} \cup \{r e^{2\pi\ii\frac{\ell}{p}}:1\leq \ell\leq p\}.
\end{equation}
We set $b_1:=r\in\T_r$.
Let $CP_t=\{c_j^t: 1\leq j\leq p+q\}$ be the critical points of $f_t$ and $CV_t=\{v_j^t=f_t(c_j^t):1\leq j\leq p+q\}$ the critical values, where $t=0$ or $1$. Let $\{\gamma_j^t:1\leq j\leq p+q\}$ be a set of admissible critical value curves of $f_t$, where $t=0$ or $1$.
By the definition of admissible critical value curves, there is a continuous map
\begin{equation}\label{equ:phi}
\chi:[0,1]\times \overline{\D}\to \overline{\D}
\end{equation}
such that
\begin{itemize}
\item For every $t\in[0,1]$, $\chi(t,\cdot):\overline{\D}\to\overline{\D}$ is a homeomorphism;
\item $\chi(0,\cdot)=\id$ and $\chi(t,\cdot)|_{\T}=\id$ for all $t\in[0,1]$; and
\item $\chi(1,\gamma_j^0)=\gamma_j^1$ and $\chi(1,v_j^0)=v_j^1$, where $1\leq j\leq p+q$.
\end{itemize}
For $0<t<1$ and $1\leq j\leq p+q$, we denote $\gamma_j^t:=\chi(t,\gamma_j^0)$ and $v_j^t:=\chi(t,v_j^0)$.

According to Lemma \ref{lem:twist}, the annulus $\overline{\A}_r$ has a nice partition by the preimages of admissible critical value curves. For $t=0$ or $1$, let $\beta_j^t$ be the connected component of $f_t^{-1}(\gamma_j^t)$ containing $c_j^t$, where $1\leq j\leq p+q$. Note that one end point of $\beta_1^0$  and that of $\beta_1^1$ are both $b_1=r\in\T_r$. By \eqref{equ:f0-1} and $f_0|_{\partial\A_r}=f_1|_{\partial\A_r}$, we assume that the other end points of $\beta_1^0$ and $\beta_1^1$ are $e^{2\pi\ii\frac{j_0}{q}}$ and $e^{2\pi\ii\frac{j_1}{q}}$ respectively. Then there exists $k'\in\Z$ such that $T_q^{-k}(\beta_1^0)$ is homotopic to $\beta_1^1$ in $\overline{\A}_r$ rel $\{r,e^{2\pi\ii\frac{j_1}{q}}\}$ for $k:=(j_0-j_1)+qk'$.

Note that $f_0\circ T_q^{\circ k}|_{\partial\A_r}=f_0|_{\partial\A_r}$, $f_0\circ T_q^{\circ k}:\A_r\to\D$ is also a branched covering of degree $p+q$ and $\{\gamma_j^t:1\leq j\leq p+q\}$ is a set of admissible critical value curves of $f_0\circ T_q^{\circ k}$. There exists a continuous map
\begin{equation}\label{equ:phi}
\psi:[0,1]\times \overline{\A}_r\to \overline{\A}_r
\end{equation}
such that
\begin{itemize}
\item For every $t\in[0,1]$, $\psi(t,\cdot):\overline{\A}_r\to\overline{\A}_r$ is a homeomorphism;
\item $\psi(0,\cdot)=\id$ and $\psi(t,\cdot)|_{\partial\A_r}=\id$ for all $t\in[0,1]$; and
\item For $1\leq j\leq p+q$, the following diagram is commutative:
\begin{equation}
\begin{CD}
(\overline{\A}_r;T_q^{-k}(\beta_j^0),T_q^{-k}(c_j^0)) @>\psi(1,\cdot)>> (\overline{\A}_r;\beta_j^1, c_j^1) \\
@VVf_0\circ T_q^{\circ k} V   @VVf_1 V \\
(\overline{\D};\gamma_j^0,v_j^0) @>\chi(1,\cdot)>> (\overline{\D};\gamma_j^1,v_j^1).
\end{CD}
\end{equation}
\end{itemize}
Therefore, there exists a continuous map
\begin{equation}
\Phi:[0,1]\times \overline{\A}_r\to \overline{\D}
\end{equation}
such that
\begin{itemize}
\item $\Phi(0,\cdot)=f_0\circ T_q ^{\circ k }$ and $\Phi(1,\cdot)=f_1$;
\item $\Phi(t,\cdot)|_{\partial\A_r}=f_0|_{\partial\A_r}$ for all $t\in[0,1]$; and
\item $\chi(t,\Phi(0,z))=\Phi(t,\psi(t,z))$ for all $t\in[0,1]$ and $z\in\overline{\A}_r$.
\end{itemize}
Moreover, it is easy to see that $\Phi(t,\cdot):\overline{\A}_r\to \overline{\D}$ is a branched covering of degree $p+q$ for all $t\in[0,1]$.

\medskip
Under the assumption that both $f_0$ and $f_1$ are $C^1$-quasiregular, if we choose the admissible critical value curves $\{\gamma_j^t:1\leq j\leq p+q\}$ $(t=0,1)$ such that every $\gamma_j^t$ is orthogonal to $\T$ at $\omega_j$, then the two maps $\chi:[0,1]\times\overline{\D}\to\overline{\D}$ and $\psi:[0,1]\times\overline{\A}_r\to\overline{\A}_r$  can be chosen such that
\begin{itemize}
\item $\partial\chi(t,z)/\partial z$, $\partial\chi(t,z)/\partial \overline{z}$ are continuous for $(t,z)\in[0,1]\times\D$; and
\item $\partial\psi(t,z)/\partial z$, $\partial\psi(t,z)/\partial \overline{z}$ are continuous for $(t,z)\in[0,1]\times\A_r$.
\end{itemize}
This implies that $\Phi(t,\cdot):\A_r\to\D$ is $C^1$-quasiregular for all $t\in[0,1]$ and $\partial\Phi(t,z)/\partial z$, $\partial\Phi(t,z)/\partial \overline{z}$ are continuous for $(t,z)\in[0,1]\times\A_r$.
\end{proof}

\section*{Acknowledgements}

We are grateful to Guizhen Cui, Kevin Pilgrim, Xiaoguang Wang and Yongcheng Yin for helpful discussions and comments, and to the referee for careful reading and helpful comments.

\bibliographystyle{amsalpha}
\bibliography{E:/Latex-model/Ref1}

\end{document}